\documentclass[twoside]{article}
\usepackage[accepted]{aistats2024}

\usepackage{natbib}
\usepackage[theorems,skins]{tcolorbox}

\usepackage{makecell}

\usepackage{microtype}
\usepackage{graphicx}
\usepackage{subfigure}
\usepackage{booktabs} 

\usepackage{hyperref}

\usepackage{soul}

\usepackage{amsmath}
\usepackage{amssymb}
\usepackage{mathtools}
\usepackage{amsthm}

\usepackage[capitalize,noabbrev]{cleveref}

\theoremstyle{plain}
\newtheorem{theorem}{Theorem}[section]
\newtheorem{proposition}[theorem]{Proposition}
\newtheorem{lemma}[theorem]{Lemma}
\newtheorem{corollary}[theorem]{Corollary}
\theoremstyle{definition}
\newtheorem{definition}[theorem]{Definition}
\newtheorem{assumption}[theorem]{Assumption}
\theoremstyle{remark}
\newtheorem{remark}[theorem]{Remark}

\usepackage{nicefrac}       

\usepackage{xcolor}         
\hypersetup{
	colorlinks   = true, 
	urlcolor     = blue, 
	linkcolor    = {blue!50!black}, 
	citecolor   = {blue!50!black} 
}

\usepackage[flushleft]{threeparttable}
\usepackage{colortbl} 
\usepackage{cellspace}
\usepackage{multirow}

\usepackage[colorinlistoftodos,bordercolor=orange,backgroundcolor=orange!20,linecolor=orange,textsize=scriptsize]{todonotes}

\newcommand*\colourcheck[1]{%
	\expandafter\newcommand\csname gcmark\endcsname{\textcolor{#1}{\ding{52}}}%
}
\newcommand*\colourxmark[1]{%
	\expandafter\newcommand\csname rxmark\endcsname{\textcolor{#1}{\ding{55}}}%
}
\definecolor{mygreen}{HTML}{02862a}
\definecolor{myred}{HTML}{9a0000}
\colourcheck{mygreen}
\colourxmark{myred}
\definecolor{linen}{HTML}{FAF0E6} 
\definecolor{darkteal}{RGB}{0, 110, 110}

\usepackage{preamble}

\usepackage{xspace}
\newcommand{\algname}[1]{{\sf {\color{darkteal}\footnotesize #1}}\xspace}
\newcommand{\algnamesmall}[1]{{\sf {\color{darkteal}\scriptsize #1}}\xspace}

\usepackage{hyperref}
\usepackage{url}

\begin{document}

\runningtitle{Taming Nonconvex Stochastic Mirror Descent}

\twocolumn[

\aistatstitle{Taming Nonconvex Stochastic Mirror Descent \\ with General Bregman Divergence}

\aistatsauthor{ Ilyas Fatkhullin \And Niao He}

\aistatsaddress{ETH Zürich \And ETH Zürich} 

]

\begin{abstract}
This paper revisits the convergence of Stochastic Mirror Descent (\algname{SMD}) in the contemporary nonconvex optimization setting. Existing results for batch-free nonconvex \algname{SMD} restrict the choice of the distance generating function (DGF) to be differentiable with Lipschitz continuous gradients, thereby excluding important setups such as Shannon entropy. In this work, we present a new convergence analysis of nonconvex \algname{SMD} supporting general DGF, that overcomes the above limitations and relies solely on the standard assumptions. Moreover, our convergence is established with respect to the Bregman Forward-Backward envelope, which is a stronger measure than the commonly used squared norm of gradient mapping. We further extend our results to guarantee high probability convergence under sub-Gaussian noise and global convergence under the generalized Bregman Proximal Polyak-{\L}ojasiewicz condition. Additionally, we illustrate the advantages of our improved \algname{SMD} theory in various nonconvex machine learning tasks by harnessing nonsmooth DGFs. Notably, in the context of nonconvex differentially private (DP) learning, our theory yields a simple algorithm with a (nearly) dimension-independent utility bound. For the problem of training linear neural networks, we develop provably convergent stochastic algorithms. 

\end{abstract}

\section{INTRODUCTION}

We consider stochastic composite optimization
\begin{equation}\label{eq:composite_problem}
	\min_{x\in\cX}  \Exp{f(x, \xi)} + r(x) ,
\end{equation}
where $F(x) := \Exp{f(x, \xi)}$ is differentiable and (possibly) nonconvex, $r(\cdot)$ is convex, proper and lower-semi-continuous, $\cX$ is a closed convex subset of $\R^d$. The random variable (r.v.) $\xi$ is distributed according to an unknown distribution $P$. We denote $\Phi := F + r$ and let $\Phi^* := \inf_{x\in \cX}\Phi(x) > -\infty$. 

A popular algorithm for solving \eqref{eq:composite_problem} is Stochastic Mirror Descent (\algname{SMD}), which has an update rule 
\begin{equation}\label{eq:SMD}
	x_{t+1} = \argmin_{x\in \cX}  \stepsize_t(\langle \nabla f (x_t, \xi_t), x \rangle + r(x) ) + D_{\omega}(x, x_t)  , 
\end{equation}
where $D_{\omega}(x, y) $ is the Bregman divergence between points $x, y \in \cX$ induced by a distance generating function (DGF) $\omega(\cdot)$; see Section~\ref{sec:prelim} for the definitions. When $r(\cdot) = 0$, $\cX = \R^d$ and $\omega(x) = \frac{1}{2}\sqnorm{x}_2$, we have $D_{\omega}(x, y) = \frac{1}{2}\sqnorm{x - y}_2$, and \algname{SMD} reduces to the standard Stochastic Gradient Descent (\algname{SGD}). However, it is often useful to consider more general (non-Euclidean) DGFs.

\algname{SMD} with general DGF was originally proposed in the pioneering work of \citet{nemirovskij_yudin_1979_eff,nemirovskij1983problem}, and later found many fruitful applications \citep{ben2001ordered,shalev2012online,arora2012multiplicative} leveraging nonsmooth instances of DGFs. In the last few decades, \algname{SMD} has been extensively analyzed in the convex setting under various assumption, e.g., \citep{beck2003mirror,lan2012optimal,allen2014linear,birnbaum2011distributed}, including relative smoothness \citep{lu2018relatively,bauschke2017descent,dragomir2021optimal,hanzely2021accelerated} and stochastic optimization \citep{lu2019relative,nazin2019algorithms,zhou2020convergence,hanzely2021fastest,vural2022mirror,liu2023high,nguyen2023improved}. However, despite the vast theoretical progress, convergence analysis of nonconvex \algname{SMD} with general DGF still remains elusive. 

\subsection{Related Work}
We now discuss the related work in the nonconvex stochastic setting. In the unconstrained Euclidean case, \citet{ghadimi2013stochastic} propose the first non-asymptotic analysis of nonconvex \algname{SGD}. Later, \citet{ghadimi2016mini_batch} consider the more general composite problem \eqref{eq:composite_problem} with arbitrary convex $r(\cdot)$, $\cX$, and propose a modified algorithm using large mini-batches. Unfortunately, the use of large mini-batch appears to be crucial in the proof proposed in \citep{ghadimi2016mini_batch} even in Euclidean setting. Later, \citet{davis2019stoch_model_based_WC} address this issue by proposing a different analysis for \algname{Prox-SGD} (method \eqref{eq:SMD} with $\omega(x) = \frac{1}{2} \sqnorm{x}_2)$. Their elegant proof, using the notion of the so-called Moreau envelope, allows them to remove the large batch requirement in the Euclidean setting. However, their analysis crucially relies on the use of Euclidean geometry and appears difficult to extend to the more general nonsmooth DGFs of interest. In particular, the subsequent works \citep{zhang2018convergence,davis2018stochastic_HO_Gr} do consider more general DGF and derive convergence rates for \eqref{eq:SMD}. However, both works assume a smooth DGF to justify their proposed convergence measures, see our Section~\ref{sec:conv_FOSP_exp} for a more detailed comparison. Another line of work uses momentum or variance reduced estimators, e.g., \citep{zhang2021variance,huang2022bregman,fatkhullin2023momentum,ding2023_NC_MSBPG}, but agian their analysis is limited to the Euclidean geometry.

\subsection{Contributions}
\begin{itemize}
	\item In this work, we develop a new convergence analysis for \algname{SMD} under the general assumptions of relative smoothness and bounded variance of stochastic gradients.  Importantly, unlike the prior work, our analysis naturally accommodates general nonsmooth DGFs, including the important case of Shannon entropy. 
	Moreover, our analysis (i) works for any batch size, (ii) does not require the bounded gradients assumption, (iii) supports any closed convex set $\cX$, and (iv) guarantees convergence on a strong stationarity measure -- the Bregman Forward-Backward envelope.
	\item We further demonstrate the flexibility of our proof technique by extending it in two directions. First, we perform a high probability analysis under the sub-Gaussian noise improving upon the previously known rates under weaker assumptions. Next, we establish the global convergence in the function value for \algname{SMD} under the generalized version of the Proximal Polyak-{\L}ojasiewicz condition. In both cases, when specialized to the unconstrained Euclidean setup, our rates can recover the state-of-the-art bounds, up to small absolute constants.  
	
	\item Finally, we demonstrate the importance of our general theory in various machine learning contexts, including differential privacy, policy optimization in reinforcement learning, and training deep linear neural networks. For each of the considered problems, our new \algname{SMD} theory allows us to either improve convergence rates or design provably convergent stochastic algorithms. In all cases, we leverage nonsmooth DGFs to attain the result. 
	
\end{itemize}

\paragraph{Our Techniques.}
The key idea of our analysis is the use of a new Lyapunov function in the form of a weighted sum of the function value $\Phi(\cdot)$ and its Bregman Moreau envelope $\Phi_{1/\rho}(\cdot)$:
 $$
 \lambda_t := \stepsize_{t-1} \rho (\Phi(x_t) - \Phi^* )  + \Phi_{1/\rho}(x_t) - \Phi^*  .  
 $$ 
We recall that the classic analysis of (large batch) \algname{SMD} in \citep{ghadimi2016mini_batch} uses the function value as a Lyapunov function, i.e., $\lambda_{t, 1} := \Phi(x_t) - \Phi^*$. While this approach is very intuitive and matches with analysis in unconstrained case, it seems very difficult to generalize to more general constrained problem \eqref{eq:composite_problem} even in the Euclidean setting. On the other hand, the analysis pioneered in \citep{davis2019stoch_model_based_WC} uses $\lambda_{t , 2} := \Phi_{1/\rho}(x_t) - \Phi^*$ as a Lyapunov function, which does not seem straightforward to extend into non-Euclidean setups, unless the smoothness of DGF is additionally imposed. Our Lyapunov function contains a weighted average of the above two quantities, i.e., $\lambda_t = \stepsize_{t-1} \rho \lambda_{t, 1} + \lambda_{t, 2} $, where $\cb{\stepsize_t}_{t\geq0}$ is the step-size sequence (with $\stepsize_{-1} = \stepsize_0$), $\rho > 0$. This modified Lyapunov function allows to better utilize (relative) smoothness of $F(\cdot)$ in the analysis. Namely, both upper and lower bound inequalities in Assumption~\ref{ass:rel_smooth} will be used in the proof.

\section{PRELIMINARIES}\label{sec:prelim}
We fix an arbitrary norm $\norm{\cdot}$ defined on $\cX \subset \R^d$, and denote by  $\norm{\cdot}_* := \sup_{z: \norm{z}\leq 1} \langle \cdot, z \rangle$ its dual. The Euclidean norm is denoted by $\norm{\cdot}_2$. We denote by $\delta_{\cX}$ the indicator function of a convex set $\cX$, i.e., $\delta_{\cX}(x) = 0$ if $x\in \cX$ and $+\infty$ otherwise. For a closed proper function $\Phi: \mathbb{R}^d \rightarrow \mathbb{R} \cup \{+\infty\}$ with $\textrm{dom}\, \Phi := \cb{x \in \R^d\mid \Phi(x) < +\infty}$, the Fr\'echet subdifferential at a point $x\in \R^d$ is denoted by $\partial \Phi(x)$ and is defined as a set of points $g \in \mathbb{R}^d $ such that $ \Phi(y) \geq \Phi(x)+\langle g, y-x\rangle+o(\|y-x\|), \forall y \in \mathbb{R}^d $ if $x\in \textrm{dom}\, \Phi$. We set $\partial \Phi(x) = \emptyset$ if $x\notin  \textrm{dom}\, \Phi$ \citep{davis2019proximally}.\footnote{When $\Phi = \delta_{\mathcal{X}}$, the function is convex and $\partial \delta_{\mathcal{X}}(x)$ coincides with the usual subdifferential in the convex analysis.} We denote by $\text{cl}(\cX)$ and $\text{ri}(\cX)$ the closure and the relative interior of $\cX$ respectively. 

Let $\cS \subset \R^d$ be an open set and $\omega: \text{cl}(\cS) \rightarrow \R$ be continuously differentiable on $\cS$. Then we say that $\omega(\cdot)$ is a distance generating function (DGF) (with zone $\cS$) if it is $1$-strongly convex w.r.t. $\norm{\cdot}$ on $\text{cl}(\cS)$.
We assume throughout that $\cS$ is chosen such that $\text{ri}(\cX) \subset \cS$ \citep{chen1993convergence}.\footnote{Following \citet{chen1993convergence}, we can verify that, in the Euclidean setup ($\omega(x) = \frac{1}{2}\sqnorm{x}_2$), one can set $\cS = \R^d$; in the simplex setup ($\omega(x) = \sum_{i=1}^d x^{(i)} \log x^{(i)} $), the choice $\cS = \cb{x\in \R^d \mid x^{(i)} > 0 \text{ for all } i\in[d]}$ is suitable.} For simplicity, we let $\opn{dom} r = \R^d$. The \textit{Bregman divergence} \citep{bregman1967relaxation} induced by $\omega(\cdot)$ is 
$$
D_{\omega}(x, y) := \omega(x) - \omega(y) - \langle \nabla \omega(y), x - y \rangle \quad \text{for } x, y \in \cS  . 
$$
We denote by $D_{\omega}^{\text{sym}}(x , y ) := D_{\omega}(x, y) + D_{\omega}(y, x ) $ a symmetrized Bregman divergence.

For any $\Phi : \R^d \rightarrow \R \cup \cb{+\infty}$ and a real $\rho> 0$, the Bregman Moreau envelope and the proximal operator are defined respectively by 
$$
\Phi_{1/\rho}(x) := \min_{y\in \cX} \sb{ \Phi(y) + \rho D_{\omega}(y, x) } ,
$$
$$
\operatorname{prox}_{\Phi /\rho}(x) := \argmin_{y\in \cX} \sb{ \Phi(y) + \rho D_{\omega}(y, x)  } .
$$

A point $x\in \cX \cap \cS  $ is called a first-order stationary point (FOSP) of \eqref{eq:composite_problem} if $0\in \partial (\Phi + \delta_{\cX})(x)$ for $\Phi := F + r$. 

\subsection{FOSP Measures}
We define three different measures of first-order stationarity for a candidate solution $x\in \cX \cap \cS  $. 

$(i)$ \textit{Bregman Proximal Mapping (BPM)} 
$$
 \Delta_{\rho}(x) := \rho^2 D_{\omega}^{\text{sym}}(\hat x, x)   , \quad  \hat x := \operatorname{prox}_{\Phi/\rho}(x). 
$$
$(ii)$ \textit{Bregman Gradient Mapping (BGM)} 
$$
\Delta_{\rho}^+(x) := \rho^2 D_{\omega}^{\text{sym}}(x^+, x) ,
$$
$$
x^+ := \argmin_{y \in \cX} \, \langle \nabla F(x), y \rangle + r(y) + \rho D_{\omega}(y, x) . 
$$

$(iii)$ \textit{Bregman Forward-Backward Envelope (BFBE)} 
$$
 \mathcal{D}_{\rho}(x) := - 2 \rho \min_{y\in \cX}  Q_\rho(x, y) , \\
$$
$$
Q_\rho(x, y) := \langle\nabla F(x), y - x\rangle + \rho D_{\omega}(y, x) + r(y) - r(x)  .
$$

In unconstrained Euclidean case, i.e., $\cX = \R^d$, $r(\cdot) = 0$ and $\omega(x) = \frac{1}{2}\sqnorm{x}_2$, we have $ \Delta_{\rho}^+(x) = \mathcal{D}_{\rho} (x) = \sqnorm{\nabla F(x) }_2$, which is the standard stationarity measure in non-convex optimization.\footnote{This is, however, not true for BPM, which reduces to the gradient norm of a surrogate loss, i.e., $\Delta_{\rho}(x) = \sqnorm{\nabla F_{1/\rho}(x)}_2$ for $\rho > \ell$, where $\ell$ is a smoothness constant of $F(\cdot)$.} 
Note that all three quantities presented above are measures of FOSP in the sense that if one of them $ \Delta_{\rho}( x)$ ,  $ \Delta_{\rho}^+(  x)$ or  $  \mathcal{D}_{\rho} (x) $ is zero for some $x\in \cX \cap \cS $, then $0 \in \partial (\Phi + \delta_{\cX})( x) $. However, it is more practical to understand what happens if one of them is only $\varepsilon$-close to zero. In Section~\ref{sec:connections} we establish the connections between these quantities for any $x\in \cX \cap \cS$, and find that BFBE, $\cD_\rho(x)$, is the strongest among the three. We should mention that the use of BFBE is not new for the analysis of optimization methods. In the Euclidean case, BFBE was initially proposed in \citep{patrinos2013proximal_Newton}, and its properties were later analyzed in \citep{stella2017forward,liu2017further_porp_FBE}. Later, \citet{ahookhosh2021bregman} consider BFBE in general non-Euclidean setting. 
However, to our knowledge it was not considered in the context of stochastic even in the Euclidean setup.

\section{ASSUMPTIONS}
Throughout the paper we make the following basic assumptions on $F(\cdot)$ and the stochastic gradients.
\begin{assumption}[Relative smoothness \citep{bauschke2017descent,lu2018relatively}]\label{ass:rel_smooth}
	A differentiable function $F: \cX \cap \cS\rightarrow \R$ is said to be $\ell$-\textit{relatively smooth} on $\cX \cap \cS$ with respect to (w.r.t.) $\omega(\cdot)$ if for all $x, y \in \cX \cap \cS$ 
	$$
	- \ell D_{\omega}(x,y) \leq F(x) - F(y) - \langle \nabla F(y), x - y \rangle \leq \ell D_{\omega}(x,y)  . 
	$$
 We denote such class of functions as $(\ell, \omega)$-smooth.
\end{assumption}

It is known that smoothness w.r.t.$\,\norm{\cdot}$, i.e., $\norm{\nabla F(x) - \nabla F(y)}_* \leq \ell \norm{x - y}$ for all $x, y \in \cX \cap \cS$,
implies 
Assumption~\ref{ass:rel_smooth} \citep{nesterov2018lectures}. 

\begin{assumption}\label{ass:BV}
	We have access to a stochastic oracle that outputs a random vector $\nabla f(x, \xi)$ for any given $x\in\cX$, such that $\Exp{ \nabla f (x, \xi) } = \nabla F(x), $
	\begin{equation*}
		\Exp{ \sqnorm{ \nabla f(x, \xi) - \nabla F(x) }_* } \leq \sigma^2 ,
	\end{equation*}
where the expectation is taken w.r.t. $\xi \sim P$.
\end{assumption}

\section{MAIN RESULTS}
\subsection{Connections between FOSP Measures}\label{sec:connections}
We start by establishing the connections between introduced convergence measures. It turns out that BPM and BGM are essentially equivalent, i.e., differ only by a small (absolute) multiplicative constant. 
\begin{lemma}[BPM $\approx$ BGM]\label{le:BPMandBGM}
	Let $F(\cdot)$ be $(\ell, \omega)$-smooth and $\sqrt{D_{\omega}^{\text{sym}}(x, y)}$ be a metric. Then for any $x \in \cX \cap \cS $, and $\rho, s > 0$ such that $ \rho  > \ell / s +  2 \ell $, it holds
	$$
	\frac{ \Delta_{\rho}(x)}{C(\ell, \rho, s)} \leq \Delta_{\rho}^+ (x) \leq C(\ell, \rho, s) \Delta_{\rho}(x) , $$
	where $C(\ell, \rho, s) := \frac{(1+s)(\rho - \ell) + (1 + s^{-1}) \ell }{\rho - \ell - (1+s^{-1}) \ell }$.  In particular, for $s = 1 $, $\rho = 4 \ell$, we have $C(\ell, \rho, s) = 8$, and
	$$
	\frac{1}{8} \Delta_{4\ell}(x) \leq  \Delta_{4\ell}^+(x) \leq 8 \Delta_{4\ell}(x) .
	$$
\end{lemma}
This result is in a similar spirit to Theorem 4.5 in \citep{drusvyatskiy2019efficiency}. However, their proof only works in the Euclidean setting and does not readily extend to other DGFs. Our proof is different and can accommodate a possibly nonsmooth DGF.\footnote{Unfortunately, it is unclear if the above result holds for arbitrary $\omega(\cdot)$ that does not induce a metric. Note that, in general, DGF might not induce a metric even for popular choices of $\omega(\cdot)$. For instance, the Shannon entropy induces $\sqrt{D_{\omega}^{\text{sym}}(x, y)}$ that does not satisfy the triangle inequality, see, e.g., Theorem 3 in  \citep{acharyya2013bregman} for details.} Next, we examine the relation between BGM and BFBE. 

\begin{lemma}[BFBE $>$ BGM]\label{le:FB_env}
	For any $x\in \cX \cap \cS$ 
	$$2 \cD_{\rho/2}(x) \geq  \Delta_{\rho}^+(x). $$ 
	There is an instance of problem \eqref{eq:composite_problem} with $\ell = 1$, $\cX = [0,1]$ and $\argmin_{y\in \cX}\Phi(x) = 0$ such that for any $\rho \in [1,2]$, $\rho_1 \geq 1$ and $x \in (0,1]$ it holds 
	$$
	\frac{\cD_{\rho}(x)}{\Delta_{\rho_1}^+(x) } \geq \frac{2}{|x|} . 
	$$
\end{lemma}

The above lemma implies that BFBE is a strictly stronger convergence measure than previously considered BGM and BPM. Moreover, the difference between BFBE and BGM can be arbitrarily large even when $x$ is close to the optimum! This effect is actually very common and happens already in the Euclidean case with classical regularizer $r(x) = \norm{x}_1$. The explanation for this phenomenon is simple. Notice that BGM is defined in the primal terms, i.e., the squared distance between $x$ and $x^+$, while BFBE is defined in the functional terms (the minimum value of $Q_{\rho}(x,y)$ over $y$). Therefore, BFBE unlike BGM scales with the value of $r(x) = |x|$ rather than $x^2 $.  

We conclude from Lemma~\ref{le:BPMandBGM} and \ref{le:FB_env} that $\cD_{\rho}(x)$ is the strongest convergence measure among the three. In the subsequent sections we aim to establish convergence of \algname{SMD} directly w.r.t.\,BFBE instead of using BGM or BPM.

\subsection{Convergence to FOSP in Expectation}\label{sec:conv_FOSP_exp}
We start with our key result, which establishes convergence of \algname{SMD} in expectation. 
\begin{tcolorbox}[colback=gray!5!white,colframe=gray!75!black]
	\begin{theorem}\label{thm:SMD}
	Let Assumptions~\ref{ass:rel_smooth} and \ref{ass:BV} hold. Let the sequence $\cb{\eta_t}_{t\geq0}$ be non-increasing with $\eta_0 \leq 1/(2\ell)$, and $\bar{x}_T $ be randomly chosen from the iterates $x_0, \ldots, x_{T-1}$ with probabilities $p_t = \stepsize_t / \sum_{t=0}^{T-1}\stepsize_t$. Then 
	\begin{equation}\label{eq:main_diminishing_sz}
		\Exp{ \cD_{3\ell}(\bar x_T) } \leq \frac{3 \lambda_0 + 6 \ell \sigma^2  \sum_{t=0}^{T-1} \eta_t^2 }{ \sum_{t=0}^{T-1} \eta_t } , 
	\end{equation}
	where $\lambda_0 := \Phi_{1/\rho}(x_0) - \Phi^* + \Phi(x_0) - \Phi^*$. If we set constant step-size $\stepsize_t =\min\cb{\frac{1}{2\ell}, \sqrt{\frac{\lambda_0}{\sigma^2 \ell T}}}$, then 
	$$
	\Exp{ \cD_{3\ell}(\bar x_T) } = \cO\rb{\frac{\ell \lambda_0}{T} + \sqrt{\frac{\sigma^2 \ell \lambda_0}{T}} }.
	$$
\end{theorem}
\end{tcolorbox}
\begin{proof}[Proof sketch:] We start with 

\textbf{Step I. Deterministic descent w.r.t.\,BFBE.} We show that for any $\rho_1 \geq \rho + \ell$ 
$$
  \Phi_{1/\rho}( x) \leq  \Phi\left(x\right) - \fr{1}{2 \rho_1} \cD_{\rho_1}( x  )  .
$$
This inequality corresponds to deterministic descent on $\Phi(\cdot)$ of the Bregman Proximal Point Method. It will be useful in the next step to derive a recursion on $\cD_{\rho_1}( x  )$.

\textbf{Step II. One step progress on the Lyapunov function.} This step is the most technical one and consists of showing a progress on a carefully chosen Lyapunov function $ \lambda_t := \Phi_{1/\rho}(x_t) - \Phi^* + \stepsize_{t-1} \rho (\Phi(x_t) - \Phi^* ) $, where $\stepsize_{-1} = \stepsize_0$, $\rho > 0$: 
\begin{align*}
\lambda_{t+1} &\leq \lambda_t - \fr{\stepsize_t \rho }{2 (\rho + \ell) } \cD_{\rho + \ell}( x_t  )  +  \rho   \stepsize_t  \langle \psi_t , \hat x_t - x_{t} \rangle  \\
&  +  \rho (  \stepsize_t  \langle \psi_t , x_t - x_{t+1}\rangle   - (1-\stepsize_t \ell) D_{\omega}(x_{t+1}, x_t) ) ,
\end{align*}
where $\psi_t := \nabla f(x_t, \xi_t) - \nabla F(x_t)$.  

  \textbf{Step III. Dealing with stochastic terms.} The goal of this step is to control the stochastic terms in the above inequality using Assumption~\ref{ass:BV}. 
  
  It remains to telescope and set the step-sizes to derive the final result. 
  \end{proof}
  When specialized to the unconstrained Euclidean setting, the result of Theorem~\ref{thm:SMD} recovers (up to a small absolute constant) previously established convergence bounds for \algname{SGD} \citep{ghadimi2013stochastic} (since in this case we have $\cD_{3\ell}(\bar x_T) = \sqnorm{\nabla F(\bar x_T)}_2$), which is known to be optimal \citep{arjevani2023lower,drori2020complexity,yang2023two_sides}. However, already in the composite setting (when $r(\cdot) \neq 0$), our result is stronger than previously derived bounds for \algname{Prox-SGD} \citep{davis2018stochastic_HO_Gr} because BFBE can be much larger than BPM/BGM even in the Euclidean case as we have seen in Lemma~\ref{le:FB_env}. 
  
  In the more general non-Euclidean case, compared to Theorem 2 in \citep{ghadimi2016mini_batch}, our method does not require using large batches, and our proof works for any batch size. Moreover, \citep{ghadimi2016mini_batch} relies on the stronger assumptions: smoothness and bounded variance in the primal norm. Furthermore, a much weaker convergence measure is used in \citep{ghadimi2016mini_batch}: the squared norm of the difference between $x_t$ and $x_t^+$.\footnote{Notice that $\rho^2 \sqnorm{x - x^+} \leq \Delta_{\rho}^+(x)  \leq 2 \cD_{\rho/2}(x)$, where the first inequality holds by strong convexity of $\omega(\cdot)$, and the second is due to Lemma~\ref{le:FB_env}.} 
  
  	\citet{davis2018stochastic_HO_Gr} derive convergence of \algname{SMD} w.r.t.\,the Bregman divergence between $\hat x_t$ and $x_t$, i.e., $D_{\omega}(\hat x_t, x_t)$. Such convergence measure is not satisfactory for two reasons. First, for a general DGF of interest, the Bregman divergence is not symmetric, and it can happen that $D_{\omega}(\hat x_t, x_t)$ vanishes, while $D_{\omega}( x_t, \hat x_t)$ does not (see, e.g.,  Proposition 2 in \citep{bauschke2017descent}). Second, to justify this measure the authors in \citep{davis2018stochastic_HO_Gr} assume $\omega(\cdot)$ to be twice differentiable and notice that $2 \rho^2 D_{\omega}(\hat x_t, x_t) \geq \sqnorm{ ( \nabla^2 \omega(x_t) )^{-1} \nabla \Phi_{1/\rho}(x_t) }_* $, where $ \nabla \Phi_{1/\rho} (\cdot)$ is the gradient of the Moreau envelope of $\Phi(\cdot)$. However, the latter measure also does not seem to be sufficient either: even if we additionally assume the uniform smoothness of $\omega(\cdot)$, it is unclear how $ \nabla \Phi_{1/\rho} (\cdot)$ is connected to the standard convergence measures such as the gradient mapping in non-Euclidean setting. In the concurrent work to \citep{davis2018stochastic_HO_Gr}, \citet{zhang2018convergence} derive convergence of \algname{SMD} on the BPM. They also notice that if $\omega(\cdot)$ is differentiable and smooth (i.e.,\,$\nabla \omega(\cdot)$ is $M$-Lipschitz continuous) on $\cX$, then $\text{dist}^2(0, \partial (\Phi + \delta_{\cX}) (\hat x_t )) \leq M \Delta_{\rho}(x_t) $. However, we argue that such assumption is very strong since commonly used DGFs such as Shannon entropy are not smooth. Moreover, the analysis in \citep{zhang2018convergence} uses bounded gradients (BG) assumption, which fails to hold even for a quadratic function if $\cX$ is unbounded.\footnote{Not saying about the general relatively smooth functions, for which BG can fail even on a compact domain.}

  \subsection{High Probability Convergence to FOSP under Sub-Gaussian Noise}
  While convergence in expectation for a randomly selected point $\bar x_T$ is classical and widely accepted in stochastic optimization, it does not necessarily guarantee convergence for a single run of the method. In this section, we extend our  Theorem~\ref{thm:SMD} to guarantee convergence for a single run of \algname{SMD} with high probability. To obtain high probability bounds, we replace our Assumption~\ref{ass:BV} with the following commonly used \q{light tail} assumption on the stochastic noise distribution. 
  \begin{assumption}\label{ass:subgauss}
  	We have access to a stochastic oracle that outputs a random vector $\nabla f(x, \xi)$ for any given $x\in \cX$, such that $\Exp{\nabla f(x, \xi) } = \nabla F(x)$, and 
  	$$
  	\norm{\nabla f(x, \xi) - \nabla F(x) }_* \quad \text{is $\sigma$-sub-Gaussian r.v. } \footnote{A random variable $X$ is called $\sigma$-sub-Gaussian if $\Exp{ \exp(\lambda^2 X^2 )} \leq \exp(\lambda^2 \sigma^2 )$ for all $\lambda \in \R$ with $|\lambda| \leq 1/\sigma$.} 
  	$$
  \end{assumption}
  
  \begin{tcolorbox}[colback=gray!5!white,colframe=gray!75!black]
  \begin{theorem}\label{thm:high_prob_subgauss}
  	Let Assumptions~\ref{ass:rel_smooth} and \ref{ass:subgauss} hold. Let the sequence $\cb{\eta_t}_{t\geq0}$ be non-increasing with $\eta_0 \leq 1/(2\ell)$. Then with probability at least $1 - \beta$
  	\begin{eqnarray}
  		\frac{1}{ \sum_{t=0}^{T-1} \stepsize_t } \sum_{t=0}^{T-1} \stepsize_t  \, \cD_{5\ell}(x_t) \leq 
  		\fr{ 5 \wt \lambda_0 + 60 \sigma^2 \ell  \sum_{t=0}^{T-1} \stepsize_t^2  }{2 \sum_{t=0}^{T-1} \stepsize_t } , \notag 
  	\end{eqnarray}
  	where $\wt \lambda_0 :=  3 \, ( \Phi(x_0) - \Phi^*) + 8\,  \stepsize_0 \sigma^2 \log \rb{ \nfr 1 \beta } $ .
  \end{theorem}
  \end{tcolorbox}
	To our knowledge, the above theorem is the first high probability bound for nonconvex \algname{SMD} without use of large batches. If we use large mini-batch,\footnote{Which reduces $\sigma^2$ to $\sigma^2/B$ for mini-batch of size $B$.} then the above theorem implies $\cO\rb{  \frac{ 1 }{\varepsilon^2} + \frac{\sigma^2 }{\varepsilon^2} \log \rb{\nfr 1 \beta} + \frac{\sigma^2 }{\varepsilon^4}}$ sample complexity to ensure $\min_{0\leq t\leq T-1} \cD_{5\ell}(x_t) \leq \varepsilon^2$. Compared to the bound derived in \citep{ghadimi2016mini_batch}, which is $\cO\rb{  \frac{ 1 }{\varepsilon^2} \log \rb{\nfr 1 \beta} + \frac{\sigma^2 }{\varepsilon^4} \log \rb{\nfr 1 \beta} } $\footnote{Discarding the samples for post-proccesing step in equation (71) therin.}, our sample complexity is better by a factor of $\log \rb{\nfr 1 \beta}$. Moreover, our Assumptions~\ref{ass:rel_smooth} and \ref{ass:subgauss} are weaker than in \citep{ghadimi2016mini_batch}. When specialized to the Euclidean setup and setting the specific step-size sequences, our Theorem~\ref{thm:high_prob_subgauss} can recover (up to an absolute constant) recently derived high probability bounds for nonconvex \algname{SGD} \citep{liu2023high}. However, unlike \citep{liu2023high}, our theorem holds for any square summable step-sizes and accommodates more general (non-Euclidean) norm in Assumption~\ref{ass:subgauss}. We will demonstrate the crucial benefit of using non-Euclidean setup later in Section~\ref{sec:DP_learning}.

\subsection{Global Convergence under Generalized Proximal P{\L} condition}
	In this subsection, we are interested in global convergence of \algname{SMD} for structured nonconvex problems. We first introduce the following generalization of Proximal Polyak-Lojasiewicz (Prox-P{\L}) condition \citep{polyak1963gradient,lojasiewicz63,Lezanski63}. 
  \begin{assumption}[$\alpha$-Bregman Prox-P{\L}]\label{ass:KL}
  	There exists $\alpha \in [1,2]$ and $\mu > 0$ such that for some $\rho \geq 3 \ell$ and all $x \in \cX \cap \cS $ 
  	\begin{eqnarray}\label{eq:KL}
  		\cD_{\rho}( x  ) \geq 2 \mu (\Phi(x) - \Phi^*)^{\nfr{2}{\alpha}} .
  	\end{eqnarray}
  \end{assumption}
  The above assumption generalizes Prox-P{\L} condition studied in \citep{Karimi_PL,reddi2016proximal,li2018simple} in two ways. First, we have $\cD_{\rho}( x  )$ defined w.r.t.\,an arbitrary non-Euclidean DGF. Second, we consider $\alpha \in [1, 2]$ instead of fixing $\alpha = 2$. We will demonstrate later in Section~\ref{sec:SPG_in_RL} that both of these generalizations are important in some nonconvex problems and the flexibility of choosing $\omega(\cdot)$ can reduce the total sample complexity. We now state the global convergence of \algname{SMD}. 
  
   \begin{tcolorbox}[colback=gray!5!white,colframe=gray!75!black]
  \begin{theorem}\label{thm:SGD_KL}
  	Let Assumptions~\ref{ass:rel_smooth}, \ref{ass:BV} and \ref{ass:KL} hold. For any $\varepsilon > 0$, there exists a choice of step-sizes $\cb{\stepsize_{t}}_{t\geq 0}$ for method \eqref{eq:SMD} such that $ \min_{t\leq T}\Exp{\Phi(x_{t}^+) - \Phi^*} \leq \varepsilon$ after
  	$$
  	T = \cO\rb{ \fr{\ell \Lambda_0}{ \mu} \frac{1}{ \varepsilon^{\fr{2-\al}{\alpha}} }\log\rb{ \fr{\ell \Lambda_0  }{ \mu \varepsilon } } +  \fr{ \ell \Lambda_0 \sigma^2}{\mu^2} \frac{1}{\varepsilon^{\fr{4-\alpha}{\alpha}}}  }  . 
  	$$
  \end{theorem}
\end{tcolorbox}
    The above result implies that after at most $T$ iterations \algname{SMD} will find a point $x_t$ which is one Mirror Descent step away from a point that is $\varepsilon$-close to $\Phi^*$ in the function value. In the unconstrained Euclidean setting, the above sample complexity matches with that of \algname{SGD} \citep{KL_PAGER_Fatkhullin}.\footnote{We use a different step-size sequence $\cb{\stepsize_{t}}_{t\geq0}$ compared to $\stepsize_{t} = 1/t^{\zeta}$, $\zeta > 0$ used in \citep{KL_PAGER_Fatkhullin}, see Appendix~\ref{sec:Global_GProxPL}. This allows us to derive noise adaptive rates, i.e., if $\sigma = 0$, then we recover the iteration complexity of (deterministic) mirror descent.} In the special case $\alpha = 2$, it implies the linear convergence rate in deterministic case and $\cO\rb{\varepsilon^{-1}}$ sample complexity in the stochastic case. The linear convergence and $\cO\rb{\varepsilon^{-1}}$ sample complexity of \algname{SMD} were previously shown under relative smoothness and relative strong convexity, e.g., in \citep{lu2018relatively,hanzely2021fastest}. Our result under Assumption~\ref{ass:KL} is more general since the relative strong convexity of $F(\cdot)$ implies \eqref{eq:KL} with $\alpha=2$, see Lemma~\ref{le:rel_SC_Prox_PL}. It is also known that such rates are optimal for $\alpha = 2$ in the Euclidean setting \citep{yue2023lower,agarwal2009information}.

  \section{NEW INSIGHTS FOR MACHINE LEARNING}\label{sec:new_insights_ML}
  In this section, we dive into the context of several machine learning applications. We illustrate how each of our Theorems~\ref{thm:SMD}, \ref{thm:high_prob_subgauss} and \ref{thm:SGD_KL} can be applied to specific problems; either yielding faster convergence than existing algorithms or allowing us to design provably convergent schemes. Interestingly, the presented problems are very diverse and allow us to demonstrate different aspects of our assumptions. In all presented examples, we crucially rely on the choice of nonsmooth DGFs, which was not theoretically possible to handle in the prior work on \algname{SMD}. 
  
  \subsection{DP Learning in $\ell_2$ and $\ell_1$ Settings}\label{sec:DP_learning}
  
  In differentially private (DP) stochastic nonconvex optimization, the goal is to design a \textit{private algorithm} to minimize the population loss of type \eqref{eq:composite_problem} over a subset of a $d$-dimensional space given $n$ i.i.d.\,samples, $\xi^1, \ldots, \xi^n$, drawn from a distribution $P$. Denote by $S := \cb{\xi^1, \ldots, \xi^n}$, the sampled dataset, and by $\nabla F(x) := \sum_{i=1}^{n} \nabla f(x, \xi^i)$, the gradient of the empirical loss $ F(x) := \sum_{i=1}^{n} f(x, \xi^i)$ based on dataset $S$. The classical notion to quantify the privacy quality is  
  
  \begin{definition}[$(\epsilon, \delta)$-DP \citep{dwork2006calibrating}]\label{def:DP}
  	A randomized algorithm $\mathcal{M}$ is $(\epsilon, \delta)$-differentially private if for any pair of datasets $S, S^{\prime}$ that differ in exactly one data point and for any event $\mathcal{Y} \subseteq R a n g e(\mathcal{M})$ in the output range of $\mathcal{M}$, we have
  	$$
  	\Pr\rb{ \mathcal{M}(S) \in \mathcal{Y} } \leq e^\epsilon \Pr\rb{\mathcal{M}\left(S^{\prime}\right) \in \mathcal{Y} } + \delta ,
  	$$
  	where the probability is w.r.t.\,the randomness of $\mathcal{M}$.
  \end{definition}
  There are several common techniques to ensure privacy, which include output \citep{wu2017bolt,zhang2017efficient}, objective function \citep{chaudhuri2011differentially,kifer2012private,iyengar2019towards} or gradient perturbations \citep{bassily2014private,wang2017differentially}. Most recent works on nonconvex DP learning focus on the latter approach. The key idea of gradient perturbation is to inject an artificial Gaussian noise $b_t \sim \mathcal{N}(0, \sigma_{\text{G}}^2 I_d )$ into the evaluated gradient. 
  The parameter $\sigma_{\text{G}}^2$ should be carefully chosen to ensure privacy, which can be guaranteed by the moments accountant
  \begin{lemma}[Theorem 1 in \citep{abadi2016deep}]
  	Assume that $\norm{\nabla F(x)}_2 \leq G$ for all $x\in \cX$. There exist constants $c_1, c_2 > 0$ so that given the number of iterations $T \geq 0$, for any $\epsilon \leq c_1 T$ , the gradient method using $\nabla F(x_t) + b_t$, $b_t \sim \mathcal{N}(0, \sigma_{\text{G}}^2 I_d )$ as the gradient estimator is $(\epsilon, \delta)$-DP for any $\delta > 0$ if 
  	$
  	\sigma_{\text{G}}^2 \geq c_2 \frac{G^2 T \log\rb{\nfr 1 \delta}}{n^2 \epsilon^2} .
  	$
  \end{lemma}	
\textbf{$\ell_2$ Setting.} For instance, the \algname{DP-Prox-GD} iterates 
$$
x_{t+1} = \opn{prox}_{\stepsize_{t} r }(x_t - \eta_t ( \nabla F(x_t) + b_t ) ) , \, \, b_t \sim \mathcal{N}(0, \sigma_{\text{G}}^2 I_d ) ,
$$
where $\opn{prox}_{\stepsize_{t} r } (x) := \argmin_{y\in \cX} \rb{ r(y) + \fr{1}{2 \stepsize_{t} } \sqnorm{y -x}_2 }$. 
	Our Theorem~\ref{thm:high_prob_subgauss} immediately implies the high probability utility bound for \algname{DP-Prox-GD}:
   \begin{equation}\label{eq:DP-Prox-GD}
   	\frac{1}{T} \sum_{t=0}^{T-1} \Exp{ \cD_{5\ell}(x_t) } = \cO\rb{ \frac{ \sqrt{ d \log\rb{\nfr 1 \delta} \log(\nfr 1 \beta) }  }{n \epsilon} } ,
      \end{equation}
   where $\beta \in (0, 1)$ is the failure probability, see Corollary~\ref{cor:DP_Prox_GD} for more details and the dependence on omitted constants. 
   To our knowledge, nonconvex utility bound of \algname{DP-Prox-GD} was previously studied only in the unconstrained setting ($r(\cdot)=0$, $\cX = \R^d$), e.g.,  \citep{wang2017differentially,wang2019differentially,zhou2020private} or in expectation, e.g., \citep{wang2019differentiallyAAI}. 
   Our bound \eqref{eq:DP-Prox-GD} generalizes these works to non-trivial $\cX$ and $r(\cdot)$. 
    
  \textbf{$\ell_1$ Setting.} One issue with the above utility bound is the polynomial dimension dependence. In certain cases, this dependence can be significantly improved, e.g., when the optimization is defined on a unit simplex $\cX = \cb{x\in \R^d \vert \sum_{i=1}^d x^{(i)} \leq 1, \, x^{(i)} \geq 0}$. Notably, it makes a big difference which norm we use to measure the variance of $b_t$, e.g., $\bb E\sqnorm{b_t}_2 = d \,\sigma_G^2$ and $\bb E \sqnorm{b_t}_{\infty} \leq 2 \log(d) \, \sigma_G^2$. Therefore, using $\norm{\cdot}_{\infty}$ norm is more favorable. Motivated by this difference, we consider the differentially private mirror descent (\algname{DP-MD}):
$$x_{t+1} = \argmin_{y\in \cX}  \stepsize_t (\langle \nabla F (x_t) + b_t , y \rangle + r(y) ) +  D_{\omega}(y, x_t)    , $$ 
  where $b_t \sim \mathcal{N}(0, \sigma_{\text{G}}^2 I_d ) $ and $\omega(x) = \sum_{i=1}^d x^{(i)} \log x^{(i)} $.\footnote{It is known that such $\omega(\cdot)$ is $1$-strongly convex w.r.t.$\,\norm{\cdot}_1$ on a unit simplex \citep{beck2003mirror}.} Using our high probability guarantee Theorem~\ref{thm:high_prob_subgauss}, we can derive 
  \begin{corollary}\label{cor:DP_SMD}
  	Let $F(\cdot)$ be $(\ell, \omega)$-smooth for $\omega(\cdot)$, $\cX$ defined above, and $\norm{\nabla F(x)}_2 \leq G$ for all $x \in \cX$. Set $\eta_t = \frac{1}{2 \ell}$, $T = \frac{n \epsilon \sqrt{\ell}}{G \sqrt{\log(d) \log\rb{\nfr 1 \delta} \log \rb{\nfr 1 \beta }}}$, $\lambda_0 := \Phi(x_0) - \Phi^*$. Then \algname{DP-MD} is $(\epsilon, \delta)$-DP and with probability $1-\beta$ satisfies \footnote{The result can be easily extended to the case when only stochastic gradients $\nabla f(x_t, \xi_t^i)$ are used instead of $\nabla F(x_t)$. }  
  	$$
  	\frac{1}{T} \sum_{t=0}^{T-1}  \cD_{5\ell}(x_t)  = \cO\rb{ \frac{G \sqrt{\ell \lambda_0 \log(d) \log\rb{\nfr 1 \delta} \log\rb{\nfr 1 \beta}} }{n \epsilon} } ,
  	$$
  \end{corollary}
 
 The above result establishes a (nearly) dimension independent utility bound for \algname{DP-MD}, and improves the one of \algname{DP-Prox-GD} in \eqref{eq:DP-Prox-GD} by a factor of $\sqrt{\nfr{d}{\log(d)}}$. Several previous works in DP learning literature have shown the improved dimension dependence in $\ell_1$ setting, e.g.,  \citep{Asi_privateSCO_L12021,gopi2023private,bassily2021non_Eucl,bassily2021differentially,wang2019differentiallyAAI}. However, \citet{Asi_privateSCO_L12021,gopi2023private,bassily2021non_Eucl} assume convex $F(\cdot)$, and, therefore, are not directly comparable with our result. \citet{bassily2021differentially}, and \citet{wang2019differentiallyAAI} obtain nonconvex utility bounds in expectation, however, their techniques are different. Both above mentioned works rely on the linear minimization oracle and derive convergence on the Frank-Wolfe (FW) gap.\footnote{At least when restricted to Euclidean setting, FW gap is a weaker convergence measure than BFBE, see Lemma~\ref{le:FB_env} and \ref{le:FWGap_BGM}.} Moreover, \citet{bassily2021differentially} use a complicated double loop algorithm based on momentum-based variance reduction technique.  
    
  \subsection{Policy Optimization in Reinforcement Learning (RL)}\label{sec:SPG_in_RL}
  
  Consider a discounted Markov decision process (DMDP) $M=\{\mathcal{S}, \mathcal{A}, \mathcal{P}, R, \gamma, p\}$. Here $\mathcal{S}$ is a state space with cardinality $|\mathcal{S}|$; $\mathcal{A}$ is an action space with cardinality $|\mathcal{A}|$; $\mathcal{P}$ is a transition model, where $\mathcal{P}(s^{\prime} | s, a)$ is the transition probability to state $s^{\prime}$ from a given state $s$ when action $a$ is applied; $R : \mathcal{S}\times \mathcal{A} \rightarrow [0, 1]$ is a reward function for a state-action pair $(s, a)$; $\gamma \in [0,1)$ is the discount factor; and $p$ is the initial state distribution. Being at state $s_h \in \mathcal{S}$ an RL agent takes an action $a_h \in \mathcal{A}$ and transitions to another state $s_{h+1}$ according to $\mathcal{P}$ and receives an immediate reward $r_h\sim R(s_h, a_h)$. A (stationary) policy $\pi$ specifies a (randomized) decision rule depending only on the current state $s_h$, i.e., for each $s \in \mathcal S$, $\pi_s  \in \Delta(\mathcal A)$ determines the next action $a \sim \pi_s$, where $\Delta(\mathcal A) : = \cb{\pi_s \in \R^{|\mathcal A|} \vert  \sum_{s \in \mathcal{S}} \pi_{s a} = 1 , \, \pi_{s a} \geq 0 \text{ for all } a \in \mathcal{A} }$ denotes the probability simplex supported on $\mathcal A$. The goal of RL agent is to maximize 
  \begin{eqnarray}\label{eq:RL_problem}
  V_{p}^+(\pi) := \Exp{\sum_{h=0}^{\infty} \gamma^h r_h } , \quad \pi \in \cX := \Delta(\mathcal A)^{|\mathcal S |} ,
    \end{eqnarray}
  where expectation is w.r.t.\,the initial state distribution $s_0 \sim p $, the transition model $\mathcal{P}$ and the policy $\pi$. We define $V_{p}(\pi) :=  - V_{p}^+(\pi) $ and adopt the minimization formulation of DMDP, i.e.,  $\min_{\pi \in \cX} V_{p}(\pi)$.
  
  It is known that $V_{p}(\pi)$ is smooth, but nonconvex in $\pi$. Moreover, a property similar to Proximal P{\L} (Assumption~\ref{ass:KL}) was recently established for \eqref{eq:RL_problem} \citep{agarwal-et-al21,xiao22}. That is we have for any $\pi, \pi' \in \cX$:
  $$
  \left\|\nabla V_p(\pi)-\nabla V_p\left(\pi^{\prime}\right)\right\|_{2, 2} \leq L_F \left\|\pi - \pi^{\prime}\right\|_{2, 2} ,
  $$
  \begin{equation}\label{eq:VaGD}
  V_p(\pi)-V_p^{\star} \leq C \max_{\pi^{\prime} \in \cX} \left\langle\nabla V_\mu(\pi), \pi-\pi^{\prime}\right\rangle , 
    \end{equation}
  where $L_F := \frac{2 \gamma |\mathcal A| }{(1-\gamma)^3}$, $C := \frac{1}{1-\gamma}\left\|\frac{d_p\left(\pi^{\star}\right)}{\mu}\right\|_{\infty}$, $\norm{\cdot}_{2,2}$ denotes the Frobenius norm (Lemma 4 and 54 in \citep{agarwal-et-al21}).
  
  Therefore, this problem serves well to demonstrate the application of our theory to show convergence of policy gradient (PG) methods. PG methods is the promising class of algorithms that generate a sequence of policies $\pi_t$ by evaluating the gradients $\nabla V_{\mu}(\pi_t)$ (or their stochastic estimates $\widehat\nabla V_{\mu}(\pi_t)$), where $\mu \in \Delta(\mathcal A)$ is some distribution (not necessarily equal to $p$). One of the most basic variants is the 
  
  \textbf{Projected Stochastic Policy Gradient:}
  $$
\text{\algname{P-SPG}:} \quad  \pi_{t+1} = \opn{proj}_{\cX}\rb{ \pi_t - \eta_t \widehat\nabla V_{\mu}(\pi_t) } ,
  $$
  where $\opn{proj}_{\cX}(\cdot)$ denotes the Euclidean projection onto $\cX$. Given that the variance of stochastic gradients $\widehat\nabla V_{\mu}(\pi_t)$ is bounded in the Euclidean norm by $\sigma_F^2$,\footnote{The variance of $\widehat \nabla V_{\mu}(\cdot)$ can be bounded under reasonable assumptions or using appropriate exploration strategies, e.g., $\epsilon$-greedy or Boltzmann, see \citep{daskalakis2020independent,cesa2017boltzmann,xiao22,Johnson_Opt_Conv_PG_2023}.}  our Theorems~\ref{thm:SMD} and \ref{thm:SGD_KL} imply the following 
 
   \begin{corollary}\label{cor:PSPG}
  	For any $\varepsilon >0$,  \algname{P-SPG} guarantees:  
  	
  	(i) $
  	\min_{0\leq t\leq T-1}\Exp{ \cD_{3 L_F }(\pi_t) } \leq \varepsilon^2$ after
  	$$T = \cO\rb{\frac{|\mathcal A| }{(1-\gamma)^3 \varepsilon^2} + \frac{\sigma_F^2 |\mathcal A| }{(1-\gamma)^3 \varepsilon^4} } , $$
  	
  	(ii) $\min_{t\leq T}\Exp{V_{p}(\pi_t^+) - V_{p}^*} \leq \varepsilon$ after 
  	$$
  	T =  \wt \cO\rb{ \frac{ |\mathcal A | |\mathcal S | }{ (1-\gamma)^{5} \varepsilon} +  \fr{ \sigma_F^2 |\mathcal A |^2 |\mathcal S |^2 }{(1-\gamma)^7 \varepsilon^3}   }  . 
  	$$
  \end{corollary}
  
 Convergence of \algname{P-SPG} was studied (in deterministic case) in \citep{agarwal-et-al21} using the notion of gradient mapping. Recently, an improved analysis was provided in \citep{xiao22} with iteration complexity $
  T =  \cO\rb{ \frac{ |\mathcal A | |\mathcal S |  }{ (1-\gamma)^{5} \varepsilon}  } $ to achieve $V_p(\pi_T) - V_p^* \leq \varepsilon$. If $\sigma_F = 0$, our iteration complexity in $(ii)$ recovers the one in \citep{xiao22}, albeit with a different proof.

 \textbf{Improving dependence on $|\mathcal A|$.}
  Notice that the above sample complexity bounds depend on the cardinality of the action space, which can be large in practice. The key reason for this is that the analysis of \algname{P-SPG} (\algname{Prox-SGD}) requires to measure the smoothness constant $L_F $ of $V_{p}(\pi)$ in the Euclidean (Frobenius) norm, which inevitably depends on the cardinality of the action space $|\mathcal A|$.  
  Let us instead consider $(2, 1)$-matrix norm $\norm{\cdot}_{2,1}$, i.e., $\|\pi \|_{2,1}^2 = \sum_{s\in \mathcal S} \rb{ \sum_{a\in \mathcal A} |\pi_{s a} | }^2$.\footnote{Its dual satisfies $\|\pi \|_{2,\infty}^2 = \sum_{s\in \mathcal S} \rb{ \max_{a\in \mathcal A} |\pi_{s a} | }^2$.}
  Now, we show that the dependence on $|\mathcal A|$ in the smoothness constant can be completely removed if $(2,1)$-norm is used.

  \begin{proposition}\label{prop:RL_smooth_PL}
  	For any $\pi, \pi^{\prime} \in \cX$, it holds that
  	$$
  	\left\|\nabla V_p(\pi)-\nabla V_p\left(\pi^{\prime}\right)\right\|_{2, \infty} \leq \frac{2 \gamma }{(1-\gamma)^3} \left\|\pi - \pi^{\prime}\right\|_{2, 1} . 
  	$$
  \end{proposition}
  
  Consider \textbf{Stochastic Mirror Policy Gradient} (\algname{SMPG}), that is \algname{SMD} with the matrix form of Shannon entropy $\omega(\pi) := \sum_{s\in \mathcal S} \sum_{a\in \mathcal A} \pi_{s a} \log \pi_{s a}$.\footnote{It is $1$-strongly convex w.r.t. $\norm{\cdot}_{2,1}$ norm.} The stochastic gradients in \algname{SMD} are replaced by $\widehat \nabla V_{\mu}(\pi_t) := (\widehat \nabla_1 V_{\mu}(\pi_t) , \ldots, \widehat \nabla_{|S|} V_{\mu}(\pi_t) )$. Define $E_{t} := (E_t^1, \ldots, E_t^{|S|})$, then \algname{SMPG} can be written in a closed from. For all $s\in \mathcal S$
   $$\pi_{t+1} = \pi_t \odot E_{t} ,\quad 
  E_t^s := \frac{\exp\rb{ - \eta_t \widehat \nabla_s V_{\mu}(\pi_t) }}{\sum_{a \in \mathcal{A}} \exp\rb{ - \eta_t \widehat \nabla_s V_{\mu}(\pi_t) } } ,
  $$ 
  where $\odot$ denotes an element-wise multiplication of matrices and $\exp(\cdot)$ is an element-wise exponential. The sample complexity can be derived from Theorem~\ref{thm:SGD_KL} using Proposition~\ref{prop:RL_smooth_PL} under the bounded variance assumption (in dual norm $\norm{\cdot}_{2, \infty}$). 
  
  \begin{corollary}\label{cor:SMPG}
  	For any $\varepsilon>0$,  \algname{SMPG} guarantees that $
  	\min_{0\leq t\leq T-1} \Exp{ \cD_{\rho}( \pi_t) } \leq \varepsilon^2$ with $\rho := 6\gamma (1-\gamma)^{-3}$ after
  	$$T = \cO\rb{\frac{1 }{(1-\gamma)^3 \varepsilon^2} + \frac{\sigma_{2, \infty}^2 }{(1-\gamma)^3 \varepsilon^4} }. $$
  \end{corollary}
  Notice that compared to the bound for \algname{P-SPG} the above sample complexity is better at least by a factor of $|\mathcal A|$.  Moreover, $\sigma_{2, \infty}$ can be much smaller than $\sigma_{F}$. 
  
  It should be noted, however, that $\cD_{\rho}(\pi_t) $ in Corollaries~\ref{cor:PSPG} and \ref{cor:SMPG} are induced by different $\omega(\cdot)$ and thus induce different FOSP measures. Also it remains unclear how to establish a global convergence of \algname{SMPG} in the function value. The technical difficulty arises because the condition \eqref{eq:VaGD} might not imply Assumption~\ref{ass:KL} under non-smooth DGF. 
  
  \begin{remark}
	While this example serves well to illustrate the application of our general theory and potential advantages of \algname{SMPG} compared to \algname{P-SPG}, it does not mean that \algname{SMPG} is the most suitable algorithm for solving~\eqref{eq:RL_problem}. In fact, there are other specialized algorithms in RL literature, which have better theoretical sample complexities than shown above. For example, Natural Policy Gradient (\algname{NPG}) (also known as exponentiated $Q$-descent or Policy Mirror Descent) \citep{kakade2001natural,agarwal-et-al21,lan2023policy,xiao22,zhan2023policy,khodadadian2021linear} achieves faster convergence in terms of $\varepsilon$. However, a notable difference of \algname{SMPG} compared to \algname{NPG} is that the latter uses a $Q$-function instead of the policy gradient $\widehat \nabla V_\mu(\pi)$, see the derivation of \algname{NPG} in Section 4 in \citep{xiao22}. Another popular approach to problem \eqref{eq:RL_problem} is the use of soft-max policy parametrization instead of directly solving the problem over $\cX$. In this direction, different variants of PG method were developed and analyzed, see, e.g., \citep{zhang2020variational,zhang2021convergence,barakat2023reinforcement}. 
  \end{remark}
  \begin{remark}
  	The special cases and variants of Prox-P{\L} condition were previously used to derive global convergence of PG methods \citep{daskalakis2020independent,kumar2023towards} including continuous state action-spaces in RL \citep{ding2022global,SPGM_fatkhullin23a} and classical control tasks \citep{fazel-et-al18,fatkhullin2021optimizing,LQR_output_global,LearningZS_LQGames_Wu_2023}. An alternative approach to global convergence of \algname{P-SPG}, based on hidden convexity of \eqref{eq:RL_problem}, was recently studied in \citep{fatkhullin2023_HC}.
  \end{remark}
  
 \subsection{Training Autoencoder Model using SMD}\label{sec:DNN}

In this section, we showcase how we can harness general Bregman divergence in \algname{SMD} to address modern machine learning problems involving linear neural networks, where the objectives go beyond the smooth regime considered in the existing theoretical analysis \citep{kawaguchi2016deep}.

More specifically, assume that the $F(\cdot)$ is twice differentiable and its Hessian is bounded by the polynomial of $\norm{x}_2$, i.e., there exist $r, L, L_r\geq 0$ such that 
\begin{equation}\label{eq:L_Lr_condition}
\norm{\nabla^2 F(x)}_{\text{op}} \leq L + L_r \norm{x}_2^r  \qquad \text{for all } x \in \R^d . \footnote{Compare to $(L_0, L_1)$-smoothness condition studied in \citep{zhang2019gradient_clipping_L0L1}.}
\end{equation}
The following result (initially appeared in \citep{lu2018relatively,lu2019relative}) shows that for any $r\geq 0$, the above condition implies relative smoothness (Assumption~\ref{ass:rel_smooth}).
\begin{proposition}[Proposition 2.1. in \citep{lu2018relatively}]\label{prop:polygrow_hess} Suppose $F(\cdot)$ is twice differentiable and satisfies \eqref{eq:L_Lr_condition}. Then $F(\cdot)$ is $\ell$-smooth relative to $\omega(x)=\frac{1}{r+2}\|x\|_2^{r+2}+\frac{1}{2}\|x\|_2^2$ with $\ell := \max\cb{L , L_r}$. 
\end{proposition}

To design a provably convergent scheme for such problems, it remains to solve \algname{SMD} subproblem with DGF specified in the above proposition. Luckily, this is possible 
\begin{eqnarray}
	c_t  &= &(1 + \norm{x_t}_2^r) \, x_t - \stepsize_{t} \nabla f(x_t, \xi_t ) , \label{eq:SMDr_line1}\\
	x _{t+1} &=& (1 + \theta_{*}^{r})^{-1} \, c_t  ,  \label{eq:SMDr_line2}
\end{eqnarray}
where $\theta_* \geq  0 $ is the unique solution to $\theta^{r+1} + \theta = \norm{c_t}_2  $.

Convergence to a FOSP of the above method follows immediately from our Theorem~\ref{thm:SMD}, see Appendix~\ref{sec:DNN_appendix} for more details. For $r = 1, 2$, the solution $\theta_*$ can be found in a closed form, while for larger values of $r$ it can be solved using a bisection method. 
\begin{figure}
	\centering
	\includegraphics[width=0.45\textwidth]{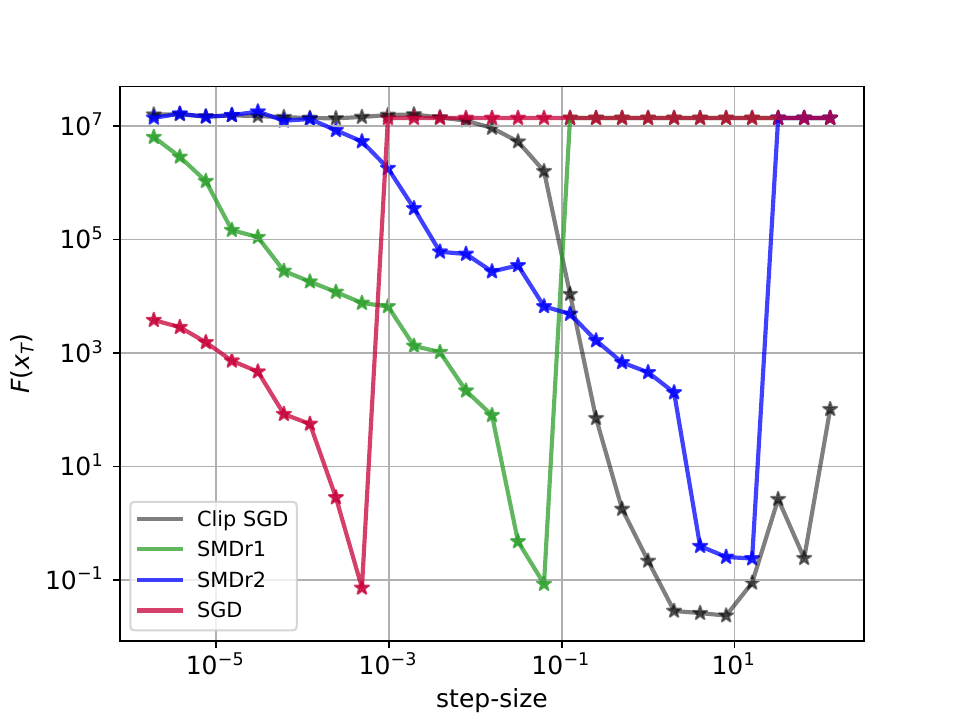}
	\caption{\footnotesize Sensitivity to step-size choice for \algnamesmall{SMDr1}, \algnamesmall{SMDr2}, \algnamesmall{SGD} and \algnamesmall{Clip SGD} (with clipping radius $1$). The plot shows the function value $F(x_T)$ after $T=10^4$ iterations for each step-size. The star markers correspond to the actual runs, and the lines linearly interpolate between them. }
	\label{fig:SMD_sz_robustness}
\end{figure}

To illustrate the empirical performance of the above scheme, we consider two layer autoencoder problem
\begin{equation}\label{eq:autoencoder_problem}
	\small \min _{\makecell{ {W_1} \in \mathbb{R}^{d_{e} \times d_{f}}  \\ {W_2} \in \R^{d_{f} \times d_{e}}}  } \left[F(W):=\frac{1}{n} \sum_{i=1}^{n}\left\|{W_2} {W_1} a_{i}-a_{i}\right\|_2^{2}\right],
\end{equation}
where $a_i \in \R^{d_f}$ are flattened representations of images, $W = \sb{W_1, W_2}$ are learned parameters of the model.\footnote{In previous notations, we have $x = \opn{vec}(W) \in \R^{2 \, d_e \, d_f}$.} The above problem is nonconvex and globally nonsmooth in $W$ since its Hessian norm grows as a polynomial in the norm of $W$. Therefore, \algname{SGD} can easily diverge if poorly initialized. However, condition \eqref{eq:L_Lr_condition} can be verified for some $r_0 \geq 2$ and the scheme \eqref{eq:SMDr_line1}, \eqref{eq:SMDr_line2} provably converges for any $r\geq r_0$.\footnote{Computation of the Hessian and use of Cauchy-Schwarz inequality implies \eqref{eq:L_Lr_condition} for any number of layers in \eqref{eq:autoencoder_problem}. }

We focus on $r = 1, 2$ and call the corresponding methods \algname{SMDr1} and \algname{SMDr2}. We compare these algorithms to the standard \algname{SGD}, which corresponds to $r = 0$. 
For comparison, we also include a popular heuristic algorithm, \algname{Clip SGD}, which is known to mitigate the problem of exploding gradients. 
We use constant step-sizes for each method. Figure~\ref{fig:SMD_sz_robustness} reports the final training loss for each step-size in the range $\cb{2^{-19}, 2^{-18}, \ldots, 2^7 }$. 
  
  We observe from Figure~\ref{fig:SMD_sz_robustness} that \algname{SMDr1} and \algname{SMDr2} allow for much larger step-sizes than \algname{SGD}. Moreover, increasing $r$ also increases the robustness to the step-size choice, i.e., the lower part of the curve becomes wider. At the same time, in the high accuracy regime, \algname{Clip SGD} still outperforms other algorithms on this task. Note, however, that \algname{Clip SGD} might not converge in general in the stochastic setting for a constant clipping parameter, see, e.\,g., \citep{koloskova2023revisiting}.

\section*{Acknowledgments} This work is supported by ETH AI Center Doctoral Fellowship and Swiss National Science Foundation (SNSF) Project Funding No. 200021-207343.



\bibliographystyle{abbrvnat}

\bibliography{biblio}

\appendix
\onecolumn

\tableofcontents

\newpage

\section{Proofs of Lemma~\ref{le:BPMandBGM} and \ref{le:FB_env}: Connections Between FOSP}

\textbf{BPM and BGM are equivalent up to a constant factor.} 

The following lemma establishes that if $F(\cdot)$ is $(\ell, \omega)$-smooth, then the distance between $\hat x$ and $x^+$ is small. 
\begin{lemma}\label{le:BPMandBGM_close}
	For any $\rho > \ell$, we have for all $x \in \cX \cap \cS$
	$$
	\rho^2 D_{\omega}^{\text{sym}}(\hat x, x^+) \leq \frac{\ell}{\rho - \ell} \rb{ \Delta_{\rho}(x) + \Delta_{\rho}^+(x) } .
	$$
\end{lemma}
\begin{proof}
	Recall that $\hat x := \argmin_{y \in \cX} F(y) + r(y) + \rho D_{\omega}(y, x) , $ and $x^+ := \argmin_{y \in \cX} \, \langle \nabla F(x), y \rangle + r(y) + \rho D_{\omega}(y, x) .$
	By the optimality conditions for $x^+$ and $\hat x$, there exist $s^+ \in \partial (r + \delta_{\cX}) (x^+)$ and $\hat s \in \partial (r + \delta_{\cX})(\hat x)$, such that 
	$$
	0 =  \nabla F(\hat x) +  \hat s + \rho \, \nabla \omega(\hat x) - \rho \, \nabla \omega(x) ,
	$$
	$$
	0 =  \nabla F(x) +  s^+ + \rho \, \nabla \omega(x^+) - \rho \, \nabla \omega(x) , 
	$$ 
	Subtracting these equalities, we obtain
	$\rho \, \nabla \omega(\hat x) - \rho \, \nabla \omega(x^+) = s^+ - \hat s + \nabla F(x) - \nabla F(\hat x) $.
	
	By Lemma~\ref{le:bregman_properties}-\ref{le:bregman_properties_TPI} (three point identity) with $x = z$ and using the above identity, we have 
	\begin{eqnarray}
		\rho ( D_{\omega}(\hat x , x^+) + D_{\omega}(x^+ , \hat x) ) &=& \langle \rho \, \nabla \omega(\hat x) - \rho \, \nabla \omega(x^+) , \hat x - x^+  \rangle \notag \\
		&=& \langle s^+ - \hat s + \nabla F(x) - \nabla F(\hat x) , \hat x - x^+  \rangle \notag \\
		&\overset{(i)}{\leq}& \langle \nabla F(x) - \nabla F(\hat x) , \hat x - x^+  \rangle \label{eq:optimality_in_le_bregman_stat_bregman_prox}
	\end{eqnarray}
	where in $(i)$ we use convexity of $(r + \delta_{\cX})(\cdot)$. By relative smoothness of $F(\cdot)$, we have for any $x, y, z \in \cX \cap \cS$
	$$
	F(x) - F(y) - \langle \nabla F(y), x - y \rangle \leq \ell D_{\omega}(x,y) ,
	$$
	$$
	F(z) - F(x) - \langle \nabla F(x), z - x \rangle \leq \ell D_{\omega}(z,x) ,
	$$
	$$
	- \ell D_{\omega}(z, y)  \leq F(z) - F(y) - \langle \nabla F(y), z - y \rangle  .
	$$
	Adding the above inequalities gives for any $x, y, z \in \cX \cap \cS$
	$$
	\langle \nabla F(x) - \nabla F(y), y - z \rangle \leq \ell D_{\omega}(x, y) + \ell D_{\omega}(z, x) + \ell D_{\omega}(z, y).
	$$
	Applying the above inequality with $x = x$, $y  = \hat x$, $z = x^+$, we further bound \eqref{eq:optimality_in_le_bregman_stat_bregman_prox} 
	\begin{eqnarray*}
	\rho ( D_{\omega}(\hat x , x^+) + D_{\omega}(x^+ , \hat x) ) &\leq& \ell D_{\omega}(x, \hat x) + \ell D_{\omega}(x^{+}, x) + \ell D_{\omega}(x^+, \hat x) \\
	&\leq& \ell D_{\omega}(x, \hat x) + \ell D_{\omega}(x^{+}, x) + \ell D_{\omega}(x^+, \hat x) + \ell D_{\omega}(\hat x, x^+) . 
	\end{eqnarray*}
	Therefore, for any $\rho > \ell$, we have 
	\begin{eqnarray*}
		\rho^2 \rb{ D_{\omega}(\hat x , x^+) + D_{\omega}(x^+ , \hat x) } \leq \frac{\ell \rho^2 }{\rho - \ell} (D_{\omega}(x, \hat x) +  D_{\omega}(x^+, x) )  \leq  \frac{\ell}{\rho - \ell} \rb{ \Delta_{\rho}(x) + \Delta_{\rho}^+(x)}  .
	\end{eqnarray*}
\end{proof}

\begin{proof}[Proof of Lemma~\ref{le:BPMandBGM}]
	For any $x, y, z \in \cX \cap \cS$ and $s > 0$, we have  
	$$
	D_{\omega}^{\text{sym}}(x, y) \leq \rb{ \sqrt{D_{\omega}^{\text{sym}}(x, z)} + \sqrt{D_{\omega}^{\text{sym}}(y, z)} }^2 \leq \rb{1 + s} D_{\omega}^{\text{sym}}(x, z) + \rb{1 + s^{-1}} D_{\omega}^{\text{sym}}(y, z) . 
	$$
	Applying Lemma~\ref{le:BPMandBGM_close} together with the above inequality, we have 
	\begin{eqnarray*}
		\Delta_{\rho}^+(x) = \rho^2 D_{\omega}^{\text{sym}}(x, x^+) &\leq& (1+s) \rho^2  D_{\omega}^{\text{sym}}( x, \hat x)  + (1+s^{-1}) \rho^2  D_{\omega}^{\text{sym}}(x^+, \hat x)  \\
		&\leq& (1+s) \rho^2  D_{\omega}^{\text{sym}} (x,\hat x)  + \frac{ (1+s^{-1}) \ell}{\rho - \ell}  \rb{\Delta_{\rho}(x) + \Delta_{\rho}^+(x) } \\
		&\leq& \rb{ 1+s + \frac{ (1+s^{-1}) \ell}{\rho - \ell} }  \Delta_{\rho}(x) + \frac{ (1+s^{-1}) \ell}{\rho - \ell}   \Delta_{\rho}^+(x) . 
	\end{eqnarray*}
	Rearranging, we obtain the upper bound on $\Delta_{\rho}^+(x)$. For the lower bound, we act similarly and derive 
\begin{eqnarray*}
	\Delta_{\rho}(x) = \rho^2 D_{\omega}^{\text{sym}}(x, \hat x) &\leq& (1+s) \rho^2  D_{\omega}^{\text{sym}}( x, x^+)  + (1+s^{-1}) \rho^2  D_{\omega}^{\text{sym}}(\hat x, x^+)  \\
	&\leq& (1+s) \rho^2  D_{\omega}^{\text{sym}} (x, x^+)  + \frac{ (1+s^{-1}) \ell}{\rho - \ell}  \rb{\Delta_{\rho}(x) + \Delta_{\rho}^+(x) } \\
	&\leq& \rb{ 1+s + \frac{ (1+s^{-1}) \ell}{\rho - \ell} }  \Delta_{\rho}^+(x) + \frac{ (1+s^{-1}) \ell}{\rho - \ell}   \Delta_{\rho}(x) . 
\end{eqnarray*}
	Combining the above two inequalities, we have 	
	$
	\frac{ \Delta_{\rho}(x)}{C(\ell, \rho, s)} \leq \Delta_{\rho}^+ (x) \leq C(\ell, \rho, s) \Delta_{\rho}(x) .
	 $

\end{proof}

\begin{remark}
	Notice that our intermediate Lemma~\ref{le:BPMandBGM_close} does not require $\omega(\cdot)$ to induce a metric and shows that $\hat x$ and $x^+$ are close if $\Delta_{\rho}(x)$ and $\Delta_{\rho}^+(x)$ are small. However, the proof of Lemma~\ref{le:BPMandBGM} crucially relies on triangle inequality. Therefore, it is unclear whether convergence of \algname{SMD} in $\Delta_{\rho}(x)$ (that was established in \citep{zhang2018convergence} under BG assumption) implies convergence in $\Delta_{\rho}^+(x) $. In our main Theorems~\ref{thm:SMD}, \ref{thm:high_prob_subgauss} and \ref{thm:SGD_KL}, we bypass this issue and directly establish convergence on $\cD_{\rho}(x)$, that is a stronger measure than $\Delta_{\rho}^+(x)$ (and stronger than $\Delta_{\rho}(x)$ if $\sqrt{	D_{\omega}^{\text{sym}}(x, y)}$ is a metric). Moreover, as we have seen in Section~\ref{sec:prelim}, $\cD_{\rho}(x)$ seems to be a more natural FOSP measure since it reduces to $\sqnorm{\nabla F(x)}$ in unconstrained case. 
\end{remark}

\textbf{BFBE is strictly larger than BGM.} 

Now we state the proof of Lemma~\ref{le:FB_env}, which consists of two parts. 

\begin{proof}[Proof of Lemma~\ref{le:FB_env}]

\textbf{1. BFBE is not smaller than BGM.} 	
	Recall that $x^+ := \argmin_{y \in \cX} \, \langle \nabla F(x), y \rangle + r(y) + \rho D_{\omega}(y, x) .$ By the optimality condition, there exists $u^+ \in \partial r(x^+)$ such that 
	$$
	0 = \nabla F(x) + \rho ( \nabla\omega (x^+) - \nabla \omega (x) ) + u^+ .
	$$
	Thus, by convexity of $r(\cdot)$
	\begin{eqnarray*}
		r(x) &\geq& r(x^+) + \langle u^+, x - x^+ \rangle  = r(x^+) + \rho\langle \, \nabla \omega (x) - \, \nabla \omega(x^+), x - x^+ \rangle   - \langle \nabla F(x) , x - x^+ \rangle \\
		& = & r(x^+) + \rho (D_{\omega}(x, x^+) + D_{\omega}(x^+, x) ) -   \langle \nabla F(x) , x - x^+ \rangle .
	\end{eqnarray*}
	Using the above inequality and the definition of $\cD_{\rho}(x) $, we derive for any $\rho, \rho_1 > 0$ 
	\begin{eqnarray*}
		\frac{1}{2 \rho_1}\cD_{\rho_1}(x) &:=&  - \min_{y\in \cX} \cb{  \langle\nabla F(x), y - x\rangle + \rho_1 D_{\omega}(y, x) + r(y) - r(x) } \\
		&=&   \langle\nabla F(x), x - x^+\rangle - \rho_1 D_{\omega}(x^+, x) + r(x) - r(x^+)  \\
		&\geq&    \rho ( D_{\omega}(x, x^+) + D_{\omega}(x^+, x) ) - \rho_1 D_{\omega}(x^+, x) \\
		&\geq&   (\rho - \rho_1) (D_{\omega}(x^+, x) + D_{\omega}(x^+, x) )  .
	\end{eqnarray*}
	Recalling the definition of $\Delta_{\rho}^+(x) $ and setting $\rho_1 = \rho/2$, it remains to conclude that $\cD_{\rho/2}(x) \geq \frac{1}{2} \Delta_{\rho}^+(x)$.

\textbf{2. BFBE can be much larger than BGM.}

The following example shows how large can be the ratio of BFBE and BGM. Consider minimizing $\Phi(x) = F(x) + r(x) $ over $\cX = [0,1] \subset \R^1$ with $F(x) = \frac{1}{2} x^2$, $r(x) = |x|$.  We can compute the proximal operator of the absolute value as $ \opn{prox}_{r / \rho} (x) = \text{sign}(x) \max \cb{0, |x| - \frac{1}{\rho}} $, where $\text{sign}(x) = 1$, if $x\geq 0$ and $\text{sign}(x) = -1$ otherwise. For any $ x \in [0, 1] $ and $\rho \geq 1$, we can compute $x^+ = \opn{prox}_{r/\rho}(x - \rho^{-1} x) = 0$. Therefore, we have 
$$
\Delta_{\rho}^+(x) = x^2 ,
$$
$$
\cD_{\rho}(x) = - 2\rho (\langle \nabla F(x), x^+ -  x \rangle + \frac{\rho}{2}(x^+ - x)^2 + |x^+| - |x| ) =  2 \rho |x| + 2\rho \rb{1 - \frac{\rho}{2}} x^2 .
$$
In particular, taking arbitrary $\rho  = \rho_1 \geq 1$ in the first equality and $\rho \leq 2$ in the second equality, we have for any $x\in (0,1]$
$$
\frac{\cD_{\rho}(x)}{ \Delta_{\rho_1}^+(x)} \geq \frac{2 \rho |x| }{x^2} \geq \frac{2}{|x|} . 
$$
We conclude that $\cD_{\rho}(x)$ can be arbitrary larger than  $\Delta_{\rho_1}^+(x)$ even when $x$ is close to the optimum $x^* = 0$. This implies that the opposite inequality in Lemma~\ref{le:FB_env} does not hold in general even when $\rho$ and $\rho_1$ are allowed to be different. Therefore, BFBE is strictly stronger convergence measure than BGM. 

\end{proof}

\newpage

\section{Proof of Theorem~\ref{thm:SMD}: Convergence to FOSP in Expectation}

\begin{proof}
\textbf{Step I. Deterministic descent w.r.t.\,Forward-Backward Envelope.}

Define $\hat x := \operatorname{prox}_{\Phi /\rho}(x)$. Notice that for any $x\in \cX \cap \cS$, we have for any $x^+ \in \cX \cap \cS$
$$
\Phi_{1/\rho}( x)  = \Phi( \hat x) + \rho D_{\omega}(\hat x, x)  \leq \Phi(x^+) + \rho D_{\omega}(x^+, x) .
$$
We set $x^+ := \argmin_{y \in \cX} \, \langle \nabla F(x), y \rangle + r(y) + \rho_1 D_{\omega}(y, x) $ with $\rho_1 \geq \rho + \ell$. Then by relative smoothness (upper bound) of $F(\cdot)$
\begin{eqnarray}\label{eq:deterministic_descent}
    \Phi_{1/\rho}( x) &\leq& \Phi\left({x}^+\right) + \rho D_{\omega}(x^+, x) \notag \\
    &=& F\left({x}^+\right)+r\left({x}^+\right) + \rho D_{\omega}(x^+, x) \notag \\  
& \leq& F\left(x\right)+\left\langle\nabla F\left(x\right), {x}^+ - x\right\rangle + \ell D_{\omega}(x^+, x) + r\left({x}^+\right) + \rho D_{\omega}(x^+, x)  \notag \\
& = &\Phi\left(x\right)+\left\langle\nabla F\left(x\right), {x}^+ - x\right\rangle + (\rho + \ell) D_{\omega}(x^+, x) + r\left(x^+\right) - r\left(x\right) \notag  \\
& = & \Phi\left(x\right) - \fr{1}{2\rho_1} \cD_{\rho_1}( x  )  + \rb{\rho + \ell - \rho_1}  D_{\omega}(x^+, x)  \notag \\
& \leq & \Phi\left(x\right) - \fr{1}{2 \rho_1} \cD_{\rho_1}( x  )  .
\end{eqnarray}
where the last equality holds by definitions of $x^+$, $ \cD_{\rho}( x  ) $ and the last step is due to condition $\rho_1 \geq \rho + \ell$. 

\textbf{Step II. One step progress on the Lyapunov function.}

   By the update rule of $x_{t+1}$ and applying \Cref{le:bregman_properties}, \cref{le:bregman_properties_optimality} with $z = x_t$, $z^+ = x_{t+1}$, $ x = \hat x_t$, we have
\begin{eqnarray}\label{eq:optimality_SMD}
\stepsize_t \langle \nabla f(x_t, \xi_t), \hat x_t - x_{t+1}\rangle + \stepsize_t (r(\hat x_t) - r(x_{t+1}) ) \geq D_{\omega}(\hat x_t, x_{t+1}) + D_{\omega}(x_{t+1}, x_t) - D_{\omega}(\hat x_t, x_t) .
\end{eqnarray}
    By the optimality of $\hat x_{t+1}$ and using the above inequality, we derive 
		\begin{align}
			 \Phi_{1/\rho}  \left(x_{t+1}\right)  &= \Phi\left(\hat{x}^{t+1}\right) + \rho D_{\omega}(\hat{x}^{t+1}, x_{t+1})  \nonumber \\
			&\leq \Phi\left(\hat{x}^{t}\right) + \rho D_{\omega}( \hat{x}^{t} , x_{t+1})  \nonumber\\
			& \overset{\eqref{eq:optimality_SMD}}{\leq} \Phi\left(\hat{x}^{t} \right)  +  \stepsize_t \rho \langle \nabla f(x_t, \xi_t), \hat x_t - x_{t+1}\rangle + \rho D_{\omega}(\hat x_t, x_t)  - \rho D_{\omega}(x_{t+1}, x_t)   + \stepsize_t \rho ( r(\hat x_t) - r(x_{t+1}) )\nonumber \\
            & =  \Phi_{1/\rho}\left( x_{t} \right) + \stepsize_t \rho  ( r(\hat x_t) - r(x_{t})   + \langle \nabla f(x_t, \xi_t), \hat x_t - x_{t}\rangle ) +  \rho   \stepsize_t  \langle \nabla f(x_t, \xi_t), x_t - x_{t+1}\rangle  \nonumber \\
            & \qquad   - \rho D_{\omega}(x_{t+1}, x_t)  + \stepsize_t \rho  ( r( x_t ) - r(x_{t+1 }) ) \nonumber .
		\end{align}
 
  We define $\lambda_t := \Phi_{1/\rho}  \left(x_{t}\right) - \Phi^* + \stepsize_{t-1} \rho ( \Phi(x_{t}) - \Phi^*) $, $\psi_t := \nabla f(x_t, \xi_t) - \nabla F(x_t)$ . Then using the above inequality
  \begin{eqnarray}
 	\lambda_{t+1} & := & \Phi_{1/\rho}  \left(x_{t+1}\right) - \Phi^* + \stepsize_{t} \rho ( \Phi(x_{t+1}) - \Phi^*)  \notag \\
 	& \leq & \Phi_{1/\rho}\left( x_{t} \right) - \Phi^*  + \stepsize_t \rho ( r(\hat x_t) - r(x_{t})   + \langle \nabla f( x_t, \xi_t ), \hat x_t - x_{t}\rangle ) \notag \\
 	&& \qquad  +  \rho  ( \stepsize_t  \langle \nabla f(x_t, \xi_t) - \nabla F(x_t) , x_t - x_{t+1}\rangle   -  D_{\omega}(x_{t+1}, x_t) ) \notag  \\
 	&& \qquad + \stepsize_t \rho ( r(x_t) + F(x_{t+1}) + \langle \nabla F(x_t) , x_t - x_{t+1}\rangle - \Phi^* ) \notag\\
 	& \overset{(i)}{\leq} & \Phi_{1/\rho}\left( x_{t} \right) - \Phi^*  + \stepsize_t \rho ( r(\hat x_t) - r(x_{t})   + \langle \nabla f(x_t, \xi_t), \hat x_t - x_{t}\rangle ) \notag \\
 	&& \qquad  +  \rho ( \stepsize_t  \langle \nabla f(x_t, \xi_t) - \nabla F(x_t) , x_t - x_{t+1}\rangle   - (1-\stepsize_t \ell) D_{\omega}(x_{t+1}, x_t) ) \notag  \\
 	&& \qquad + \stepsize_t \rho ( r(x_t) + F(x_{t}) - \Phi^* ) \notag \\
 	& = & \lambda_t   + (\stepsize_t - \stepsize_{t-1}) \rho ( \Phi(x_t) - \Phi^* ) + \stepsize_t \rho ( r(\hat x_t) - r(x_{t})   + \langle \nabla F(x_t), \hat x_t - x_{t}\rangle )  +  \rho   \stepsize_t  \langle \psi_t , \hat x_t - x_{t} \rangle  \notag \\
 	&& \qquad  +  \rho ( \stepsize_t  \langle \psi_t , x_t - x_{t+1}\rangle   - (1-\stepsize_t \ell) D_{\omega}(x_{t+1}, x_t) ) \notag  \\
 	& \overset{(ii)}{\leq} & \lambda_t  + (\stepsize_t - \stepsize_{t-1}) \rho ( \Phi(x_t) - \Phi^* )  - \fr{\stepsize_t \rho }{2 (\rho + \ell) } \cD_{\rho+\ell}( x_t  )  -  \stepsize_{t} \rho (\rho - \ell ) D_{\omega}(\hat x_t, x_t) +  \rho   \stepsize_t  \langle \psi_t , \hat x_t - x_{t} \rangle \notag \\
 	&& \qquad  +  \rho  ( \stepsize_t  \langle \psi_t , x_t - x_{t+1}\rangle   - (1-\stepsize_t \ell) D_{\omega}(x_{t+1}, x_t) )   \notag  \\
 	& \overset{(iii)}{\leq} & \lambda_t - \fr{\stepsize_t \rho }{2 (\rho + \ell) } \cD_{\rho + \ell}( x_t  )  +  \rho   \stepsize_t  \langle \psi_t , \hat x_t - x_{t} \rangle  +  \rho (  \stepsize_t  \langle \psi_t , x_t - x_{t+1}\rangle   - (1-\stepsize_t \ell) D_{\omega}(x_{t+1}, x_t) ) 
 	\label{eq:one_step_progress_Lyap} \\
 	&& \qquad - \stepsize_{t} \rho (\rho - \ell ) D_{\omega}(\hat x_t, x_t) , \notag 
 \end{eqnarray}
  where $(i)$ follows by relative smoothness (upper bound), i.e., $  F(x_{t+1})  \leq F(x_t) - \langle \nabla F(x_t), x_t - x_{t+1}\rangle + \ell D_{\omega}( x_{t+1}, x_t) $. The inequality $(ii)$ follows from relative smoothness (lower bound) of $F(\cdot)$ and \eqref{eq:deterministic_descent} (with $\rho_1 = \rho + \ell$) since
  \begin{eqnarray}\label{eq:bound_A}
  r(\hat x_t) - r(x_{t}) + \langle \nabla F(x_t), \hat x_t - x_{t}\rangle &\leq& r(\hat x_t) - r(x_{t}) + F(\hat x_t) - F(x_t) + \ell D_{\omega}(\hat x_t, x_t) \notag \\
  &=& \Phi_{1/\rho}(x_t) - \Phi(x_{t}) + (\ell - \rho ) D_{\omega}(\hat x_t, x_t)  \notag \\
  &\leq& \Phi_{1/\rho}(x_t) - \Phi(x_{t})  + (\ell - \rho ) D_{\omega}(\hat x_t, x_t)  \notag \\
  & \leq&  - \fr{1}{2 (\rho + \ell) } \cD_{\rho+\ell}( x_t  )  - (\rho - \ell ) D_{\omega}(\hat x_t, x_t)  . \notag 
  \end{eqnarray}
  The inequality $(iii)$ holds since the sequence $\cb{\eta_t}_{t\geq 0}$ is non-increasing. 

  \textbf{Step III. Dealing with stochastic terms.} Using $D_{\omega}(x_{t+1}, x_{t}) \geq \frac{1}{2}\sqnorm{x_{t+1} - x_t}$ and the bound on the variance of stochastic gradients, we have 
  \begin{eqnarray}\label{eq:bound_C}
   \Exp{ \stepsize_t  \langle \psi^t , x_t - x_{t+1}\rangle   -  (1-\stepsize_t \ell) D_{\omega}(x_{t+1}, x_t) } &\leq&  \Exp{ \stepsize_t  \langle \psi^t, x_t - x_{t+1}\rangle   - (1-\stepsize_t \ell)\fr{1}{2} \sqnorm{x_{t+1} - x_t} }  \notag \\
    &\leq& \fr{ \stepsize_t^2}{2 (1-\stepsize_t \ell) } \Exp{\sqnorm{ \psi^t }_*} \leq  \fr{ \stepsize_t^2 \sigma^2  }{2 (1-\stepsize_t \ell) } . 
  \end{eqnarray}
Define $\Lambda_t := \Exp{\lambda_t}$. Then combining \eqref{eq:one_step_progress_Lyap} with \eqref{eq:bound_C} and setting $\rho = 2\ell$, $\eta_t \leq 1/(2\ell)$, we derive for any non-increasing step-sizes $\eta_t$
\begin{eqnarray}\label{eq:descent_SMD}
\Lambda_{t+1} &\leq& \Lambda_t  - \frac{\stepsize_t \rho }{2 (\rho + \ell) }  \Exp{ \cD_{\rho+\ell}( x_t  ) } + \fr{ \rho \stepsize_t^2 \sigma^2  }{2 (1-\stepsize_t \ell) }  \leq \Lambda_t  - \frac{ \stepsize_t }{ 3 }  \Exp{ \cD_{3 \ell}( x_t  ) } + 2 \ell \stepsize_t^2 \sigma^2   . 
\end{eqnarray}

It remains to telescope and conclude the proof since
$$
\sum_{t=0}^{T-1} \stepsize_t \Exp{\cD_{3\ell}(x_t)} \leq 3 \Lambda_0 + 6\ell \sigma^2  \sum_{t=0}^{T-1} \eta_t^2 ,
$$
where the Lyapunov function in the initial point can be bounded (setting $\stepsize_{-1} = \stepsize_0$) as 
$$
\Lambda_0 = \lambda_0 =  \Phi_{1/\rho}(x_0) - \Phi^* + \stepsize_{-1} \rho ( \Phi(x_0) - \Phi^*) \leq  \Phi_{1/\rho}(x_0) - \Phi^* +  \Phi(x_0) - \Phi^* .
$$
The proof for constant step-size follows immediately from \eqref{eq:main_diminishing_sz}.


\end{proof}

\newpage

\section{Proof of Theorem~\ref{thm:high_prob_subgauss}: High Probability Convergence to FOSP under Sub-Gaussian Noise}

\begin{proof}
	We recall the definitions of $\lambda_t = \Phi_{1/\rho}  \left(x_{t}\right) - \Phi^* + \stepsize_{t-1} \rho ( \Phi(x_{t}) - \Phi^*) $, $\psi_t = \nabla f(x_t, \xi_t) - \nabla F(x_t)$, and invoke \eqref{eq:one_step_progress_Lyap} from the proof of Theorem~\ref{thm:SMD}
	\begin{eqnarray}
		\lambda_{t+1} & = & \lambda_t - \fr{\stepsize_t \rho }{2 (\rho + \ell) } \cD_{\rho + \ell}( x_t  )  +  \rho   \stepsize_t  \langle \psi_t , \hat x_t - x_{t} \rangle  +  \rho   \stepsize_t  \langle \psi_t , x_t - x_{t+1}\rangle   - \rho (1-\stepsize_t \ell) D_{\omega}(x_{t+1}, x_t)  \notag \\
		&& \qquad - \stepsize_{t} \rho (\rho - \ell ) D_{\omega}(\hat x_t, x_t) \notag \\
		& \leq & \lambda_t - \fr{\stepsize_t \rho }{2 (\rho + \ell ) } \cD_{\rho+ \ell}( x_t  )  +  \rho   \stepsize_t  \langle \psi_t , \hat x_t - x_{t} \rangle  +  \frac{\rho   \stepsize_t^2 \sqnorm{ \psi_t }_* }{2 (1-\stepsize_t \ell) }  - \stepsize_{t} \rho (\rho - \ell ) D_{\omega}(\hat x_t, x_t) .  \label{eq:one_step_progress_Lyap_hp}
	\end{eqnarray}
	Define a (normalization) scalar $w := \frac{\rho - \ell}{6 \sigma^2 \rho \eta_0} > 0$, and a sequence $S_t := \sum_{\tau=t}^{T-1} Z_\tau$, where 
	$$
	Z_t := w \rb{ \lambda_{t+1} - 
		\lambda_t + \frac{\stepsize_t \rho }{2 ( \rho + \ell) } \cD_{\rho + \ell}(x_t) }  .
	$$
	
	Now we define the filtration $\mathcal F_t = \cb{x_0, \xi_0, x_1, \ldots, \xi_{t-1}, x_{t}}$ and compute the moment generating function (MGF) of $Z_t$ for any $0\leq t \leq T-1$
	\begin{eqnarray*}
		\Exp{ \exp(Z_t) \vert \mathcal{F}_t} & = & \Exp{ \exp \rb{ w \rb{ \lambda_{t+1} - 
					\lambda_t + \frac{\stepsize_t \rho }{2 ( \rho + \ell) }  \cD_{\rho + \ell}(x_t) }  }  \vert \mathcal F_t } \\
		&\overset{\eqref{eq:one_step_progress_Lyap_hp}}{\le}& \Exp{\exp\rb{ w  \rho \stepsize_t \langle \psi^t, \hat x_t - x_t \rangle + \frac{w \rho \stepsize_t^2 \norm{\psi_t}_*^2}{2(1 - \stepsize_t \ell ) }  - w \stepsize_{t} \rho (\rho - \ell ) D_{\omega}(\hat x_t, x_t)  } \vert \mathcal F_t  } \\
		&=& \exp\rb{ - w \stepsize_{t} \rho (\rho - \ell ) D_{\omega}(\hat x_t, x_t)  } \Exp{\exp\rb{ w  \rho \stepsize_t \langle \psi^t, \hat x_t - x_t \rangle + \frac{w \rho \stepsize_t^2 \norm{\psi_t}_*^2}{2(1 - \stepsize_t \ell ) }  }  \vert \mathcal F_t } \\
		&\overset{(i)}{\le}& \exp\rb{ - w \stepsize_{t} \rho (\rho - \ell ) D_{\omega}(\hat x_t, x_t)  } \exp\rb{ 3 \sigma^2 w^2 \rho^2 \stepsize_t^2 \norm{\hat x_t - x_t}^2 + \frac{3 \sigma^2 w \rho \stepsize_t^2 }{2(1 - \stepsize_t \ell) } }  \\
		&\overset{(ii)}{\le}&  \exp\rb{ \frac{3 \sigma^2 w \rho \stepsize_t^2 }{2(1 - \eta_t \ell)} } ,
	\end{eqnarray*}
	where in $(i)$ we apply Lemma~\ref{le:lemma22_liu}, which uses that $\norm{\psi_t}_*$ is $\sigma$-sub-Gaussian. Inequality $(ii)$ holds by the fact that $\sqnorm{\hat x_t - x_t} \leq 2 D_{\omega}(\hat x_t, x_t) $, and the choice of $w$, which guarantess that $6 \sigma^2 w^2 \rho^2 \stepsize_{t}^2 \leq w \stepsize_{t} \rho (
	\rho - \ell) $ for any $t \geq 0$. Now to compute the MGF of $S_t$ we use derive
	$$
	\Exp{\exp(S_t) \vert \mathcal{F}_t } = \Exp{ \Exp{ \exp(S_{t+1} + Z_t) \vert \mathcal{F}_{t+1} }\vert \mathcal{F}_t} = \Exp{ \exp(Z_t) \Exp{ \exp(S_{t+1} ) \vert \mathcal{F}_{t+1} }\vert \mathcal{F}_t} .
	$$
	Thus, by induction we have
	$$
	\Exp{ S_0} \le \exp\rb{  \frac{3 \sigma^2 \rho w }{ 2 } \sum_{t=0}^{T-1} \frac{ \stepsize_t^2}{(1 - \stepsize_t \ell)}  } \leq \exp\rb{  3 \sigma^2 \rho w \sum_{t=0}^{T-1}  \stepsize_t^2  } .
	$$
	where the last inequality holds by the condition $\stepsize_t \leq 1/(2\ell)$. Consequently, by Markov's inequality,
	$$
	\Pr \rb{ S_0 \ge 3 \sigma^2 \rho \sum_{t=0}^{T-1} w_t \stepsize_{t}^2 + \log \rb{ \nfr 1 \beta }  } \le \beta.
	$$
	Then with probability at least $1 - \beta $, we have 
	$$
	\sum_{t=0}^{T-1} w \rb{ \lambda_{t+1} - 
		\lambda_t + \frac{\stepsize_t \rho}{2 (\rho + \ell) } \cD_{\rho + \ell}(x_t) } \leq 3 \sigma^2 w \rho \sum_{t=0}^{T-1}  \stepsize_{t}^2  + \log \rb{ \nfr 1 \beta } .
	$$
	Telescoping the above inequality, setting $\rho = 4\ell$, and dividing by the sum of step-sizes:
	\begin{eqnarray*}
\frac{1}{ \sum_{t=0}^{T-1} \stepsize_t } \sum_{t=0}^{T-1} \stepsize_t  \, \cD_{5\ell}(x_t) &\leq& \frac{  \lambda_0 + \frac{1}{w} \log \rb{ \nfr 1 \beta } +  12 \sigma^2 \ell  \sum_{t=0}^{T-1} \stepsize_t^2 }{ \frac{2}{5} \sum_{t=0}^{T-1} \stepsize_t }  \\ 
& = &  \fr{  \lambda_0 + 8 \, \stepsize_0 \sigma^2 \log \rb{ \nfr 1 \beta } + 12 \sigma^2 \ell  \sum_{t=0}^{T-1} \stepsize_t^2  }{\frac{2}{5} \sum_{t=0}^{T-1} \stepsize_t } .
	\end{eqnarray*}

It remains to bound the Lyapunov function in the initial point setting $\stepsize_{-1} = \stepsize_0$: 
$$
\lambda_0 = \Phi_{1/\rho}(x_0) - \Phi^* + \stepsize_{-1} \rho ( \Phi(x_0) - \Phi^*) \leq  \Phi_{1/\rho}(x_0) - \Phi^* + 2 ( \Phi(x_0) - \Phi^* ) .
$$

\end{proof}

\newpage

\section{Proof of Theorem~\ref{thm:SGD_KL}. Global Convergence under Generalized Proximal P{\L} Condition }\label{sec:Global_GProxPL}

\begin{proof}[Proof of Theorem~\ref{thm:SGD_KL}]
    Invoking \eqref{eq:descent_SMD}, and substituting the value of $\rho = 2\ell$ and using $\stepsize_t \leq \frac{1}{2\ell}$, we derive under Assumption~\ref{ass:KL} 
    $$
    \Lambda_{t+1} \leq \Lambda_t - \frac{ \stepsize_t  }{ 3 } \Exp{ \cD_{3 \ell}( x_t  ) }  +  2 \ell \stepsize^2 \sigma^2  \leq \Lambda_t - \frac{2 \mu \stepsize_t  }{ 3 } \Exp{ (\Phi(x_t) - \Phi^*)^{\nfr{2}{\alpha}} } +  2 \ell \stepsize^2 \sigma^2  ,
    $$
    Notice that $\Phi(x) \geq \Phi_{1/\rho}(x)$ for any $x\in \cX$, thus 
    $$
    \Exp{\Phi(x_t) - \Phi^*} \geq \frac{1}{1+\stepsize_{t-1}\rho} \Exp{ \Phi_{1/\rho}(x_t) - \Phi^* } + \frac{\stepsize_{t-1}\rho}{1+\stepsize_{t-1}\rho} \Exp{ \Phi(x_t) - \Phi^* } = \frac{\Lambda_t}{1+\stepsize_{t-1} \rho} \geq \frac{1}{2} \Lambda_t .
    $$
  By Jensen’s inequality for $z \rightarrow z^{\nfr{2}{\alpha}}$, we have $\Exp{(\Phi(x_t) - \Phi^*)^{\nfr{2}{\alpha}}} \geq  \rb{\Exp{\Phi(x_t) - \Phi^*} }^{\nfr{2}{\alpha}} $.
  Combining the above inequalities, we get
  $\Exp{(\Phi(x_t) - \Phi^*)^{\nfr{2}{\alpha}}} \geq \frac{1}{2} \Lambda_t^{\nfr{2}{\alpha}}$. Thus, we can derive a recursion
    $$
    \Lambda_{t+1} \leq \Lambda_t -  \frac{ \stepsize_t \mu }{3}  \Lambda_t^{\nfr{2}{\alpha}}  + 2 \ell   \sigma^2 \stepsize_t^2 .
    $$
    
     Assume that for $\tau = 0, \ldots, t$, we have $\Lambda_\tau \geq \varepsilon$ (otherwise we have reached $\varepsilon$-accuracy). Then
	$$
	\Lambda_{t+1}  \leq \rb{ 1  -  \frac{\stepsize_t \mu \varepsilon^{\fr{2-\al}{\al}} }{3}  } \Lambda_t  + 2 \ell  \sigma^2 \stepsize_t^2  ,
	$$
    For any $T \geq 0$, we select the step-size sequence $\stepsize_{t}$ as follows: 
    $$
    \stepsize_{t} = \begin{cases}
    	\frac{1}{2\ell}  & \text{if } t < \lceil T/2 \rceil \text{ and } T \leq \frac{6\ell}{\mu \, \varepsilon^{\fr{2-\al}{\al}}  } , \\
    	\frac{6}{\mu  \,  \varepsilon^{\fr{2-\al}{\al}} \rb{ t + \fr{12 \ell}{\mu } \varepsilon^{- \fr{2 - \al}{\al}}  - \lceil \nfr{T}{2} \rceil }}   & \text{otherwise.} 
    \end{cases}
    $$
    By Lemma~\ref{le:L3_noise_adaptive_rate}, we have 
     $$
     \Lambda_{T+1} = \cO\rb{\frac{ \ell \Lambda_0}{\mu \, \varepsilon^{\fr{2-\al}{\al}}  } \exp\rb{ - \frac{\mu \, \varepsilon^{\fr{2-\al}{\al}} T}{ \ell}} + \frac{  \ell \sigma^2}{T \mu^2 \varepsilon^{\fr{2( 2-\al ) }{\al}}  } }  .
    $$
    
	Recalling the definition of $\Lambda_t$ and using Lemma~\ref{le:MEnv_Phi_connection}, we bound $\Lambda_{T+1} \geq \Phi_{1/\rho}(x_{T+1}) - \Phi^* \geq \Phi(x_{T+1}^+) - \Phi^* $, where $ x^+ := \argmin_{y \in \cX} \, \langle \nabla F(x), y \rangle + r(y) + \ell  D_{\omega}(y, x) . $ The sample complexity to reach $ \Phi(x_{T+1}^+) - \Phi^* \leq \varepsilon$ is 
	$$
	T = \cO\rb{ \fr{\ell \Lambda_0}{ \mu} \frac{1}{ \varepsilon^{\fr{2-\al}{\alpha}} }\log\rb{ \fr{\ell \Lambda_0  }{ \mu \varepsilon } } +  \fr{ \ell \Lambda_0 \sigma^2}{\mu^2} \frac{1}{\varepsilon^{\fr{4-\alpha}{\alpha}}}  }  . 
	$$
	
\end{proof}

\newpage

\section{Proofs for Applications}

\subsection{Differentially Private Learning in $\ell_2$ and $\ell_1$ Settings}
  \begin{proof}[Proof of Corollary~\ref{cor:DP_SMD}]
	Notice that by Lemma~\ref{le:max_tail} $\norm{b_t}_{\infty}$ is $\sigma$-sub-Gaussian r.v.\,with $\sigma^2 = 2 \log(2 d) \sigma_G^2$.\footnote{Here we used the fact that $\max_{1\leq i \leq d} |\xi_i| = \max\cb{\xi_1, -\xi_1, \ldots, \xi_d, - \xi_d }$ and applied Lemma~\ref{le:max_tail} with $n = 2 d$. } We invoke the result of Theorem~\ref{thm:high_prob_subgauss} with $\stepsize_t = \stepsize_0 = \frac{1}{2\ell}$, and obtain 
	\begin{eqnarray*}
	\frac{1}{T} \sum_{t=0}^{T-1}  \cD_{5\ell}(x_t) &=& \cO\rb{  \frac{\lambda_0}{\stepsize_0 T } + \sigma^2 \stepsize_0 \ell + \frac{\sigma^2 \log \rb{\fr 1 \beta } }{  T }  } = \cO\rb{  \frac{\ell \lambda_0}{ T } +  \sigma^2 \log \rb{\fr 1 \beta }  } \\
	&=& \cO\rb{  \frac{\ell \lambda_0}{ T } +  \frac{G^2 T \log(d) \log\rb{\fr 1 \delta}}{n^2 \epsilon^2}  \log \rb{\fr 1 \beta }  }  \\
& = & \cO\rb{ \frac{G \sqrt{\ell \lambda_0 \log(d) \log\rb{\fr 1 \delta} \log\rb{\fr 1 \beta}} }{n \epsilon} } ,
		\end{eqnarray*}
	where the last equality follows by the choice of $T$. It remains to notice that $\lambda_0 = \Phi_{1/\rho}(x_0) - \Phi^* + 2 ( \Phi(x_0) - \Phi^*) \leq 3 ( \Phi(x_0) - \Phi^*) $.
\end{proof}

\begin{corollary}\label{cor:DP_Prox_GD}
	Let $F(\cdot)$ be differentiable on a convex set $\cX$ with $L$-Lipschitz continuous gradient w.r.t.\,Euclidean norm, and $\norm{\nabla F(x)}_2 \leq G$ for all $x \in \cX$. Set $\eta_t = \frac{1}{2 L }$, $T = \frac{n \epsilon \sqrt{ L }}{G \sqrt{ d  \log\rb{\nfr 1 \delta} \log \rb{\nfr 1 \beta }}}$, $\lambda_0 := \Phi(x_0) - \Phi^*$. Then \algname{DP-Prox-GD} is $(\epsilon, \delta)$-DP and with probability $1-\beta$ satisfies 
	$$
	\frac{1}{T} \sum_{t=0}^{T-1}  \cD_{5\ell}(x_t)  = \cO\rb{ \frac{G \sqrt{\ell \lambda_0 d \log\rb{\nfr 1 \delta} \log\rb{\nfr 1 \beta}} }{n \epsilon} } ,
	$$
\end{corollary}

\begin{proof}
	The proof follows the same lines as the proof of Corollary~\ref{cor:DP_SMD}. The only difference is that instead of the infinity norm of the noise, we bound the Euclidean norm, i.e., $\norm{b_t}_{2}$ is $\sigma$-sub-Gaussian r.v.\,with $\sigma^2 = d \, \sigma_G^2$.
\end{proof}

\subsection{Policy Optimization in Reinforecement Learning}

\textbf{Prox-P{\L} condition.}

Now we will verify Assumption~\ref{ass:KL} with $\alpha = 1$ holds for our RL problem. The result is similar to Lemma 5 in \citep{xiao22}. The only difference is that we have $\pi$ instead of $\pi^+$ on the left hand side of the inequality. 

\begin{lemma}\label{le:grad_dom} Let $\omega(\pi) = \frac{1}{2}\sqnorm{\pi}_{2,2} $. Then for any $\pi \in \cX$ we have 
$$
V_p(\pi)-V_p^{\star} \leq \frac{ 2 \sqrt{2 |\mathcal S| } }{1-\gamma}\left\|\frac{d_p\left(\pi^{\star}\right)}{\mu}\right\|_{\infty} \sqrt{\Delta_{\rho}^+(\pi)  } 
$$
if $\rho \geq   G_{V, \norm{\cdot}_{2, 2} } / D_{\cX, \norm{\cdot}_{2, 2}}$, where $G_{V, \norm{\cdot}_{2, 2}} := \max_{\pi \in \Pi}\sb{\norm{\nabla V_{p}(\pi) }_{2 , 2}}$.
\end{lemma}
\begin{proof}
It was shown in Lemma 4 in \citep{agarwal-et-al21} that the following (variational) gradient domination condition holds. 
$$
V_p(\pi)-V_p^{\star} \leq \frac{1}{1-\gamma}\left\|\frac{d_p\left(\pi^{\star}\right)}{\mu}\right\|_{\infty} \max_{\pi^{\prime} \in \cX} \left\langle\nabla V_\mu(\pi), \pi-\pi^{\prime}\right\rangle \quad \text{for any } \pi \in \cX .
$$
By Lemma~\ref{le:FWGap_BGM}, we have 
$$
\max_{\pi^{\prime} \in \cX} \left\langle\nabla V_\mu(\pi), \pi-\pi^{\prime}\right\rangle  \leq  \rb{ D_{\cX, \norm{\cdot}_{2,2}} + \rho^{-1} G_{V, \norm{\cdot}_{2, 2}} } \sqrt{\Delta_{\rho}^+(\pi)} \leq 2 D_{\cX, \norm{\cdot}_{2,2}} \sqrt{ \Delta_{\rho}^+(\pi)  } 
$$
where the last inequality holds for $\rho \geq   G_{V, \norm{\cdot}_{2, 2} } / D_{\cX, \norm{\cdot}_{2, 2}}$. It remains to notice that $D_{\cX, \norm{\cdot}_{2, 2}} \leq \sqrt{2 |\mathcal{S}|}$.

\textbf{Smoothness in $(2, 1)$-norm.} Proof of Proposition~\ref{prop:RL_smooth_PL}

For any $\pi, \pi^{\prime} \in \cX$, it holds that
$$
\left\|\nabla V_p(\pi)-\nabla V_p\left(\pi^{\prime}\right)\right\|_{2, \infty} \leq \frac{2 \gamma }{(1-\gamma)^3} \left\|\pi - \pi^{\prime}\right\|_{2, 1} ,
$$

The estimate of the smoothness constant follows directly from the proof of Lemma 54 in \citep{agarwal-et-al21}, since using $(2,1)$ norm we have $\sum_{a \in \mathcal A} |u_{a , s}| \leq 1$, and the perturbation $u_{a, s}$ belongs to the probability simplex $u_{a, s} \in \Delta(\mathcal A)$.

\end{proof}

\subsection{Training Autoencoder Model using SMD}\label{sec:DNN_appendix}

\textbf{Derivation of \algname{SMDr1} and \algname{SMDr2}.}
Recall the choice of DGF from subsection~\ref{sec:DNN}
$$
\omega(x)=\frac{1}{r+2}\|x\|_2^{r+2}+\frac{1}{2}\|x\|_2^2 . 
$$
 Notice that we have $\nabla \omega(x) = \norm{x}_2^{r} x + x$. The update rule of \algname{SMD} with $\cX = \R^d$, $r(x)=0$ and the above choice of DGF satisfies
$$
\stepsize_{t} \nabla f(x_t, \xi_t) + \nabla \omega(x_{t+1}) - \nabla \omega(x_{t}) = 0. 
$$ 
Define $c_t := \nabla \omega(x_t) - \stepsize_{t} \nabla f(x_t, \xi_t) = x_t - \stepsize_{t} \nabla f(x_t, \xi_t) + \norm{x_t}_2^r x_t$. Thus,  it remains to solve for $x_{t+1}$ 
\begin{equation}\label{eq:derivation_of_SMDr}
\nabla \omega (x_{t+1}) = (\norm{x_{t+1}}_2^r + 1) \, x_{t+1} = c_t ,  
\end{equation}
which is equivalent to solving the following simple univariate equation of $\theta \geq 0 $:
\begin{equation}\label{eq:univ_equation}
\theta^{r+1} + \theta = \norm{c_t}_2 . 
\end{equation}
For $r = 1, 2$, it has an explicit form solution for any $\norm{c_t}_2 . $ We have
\begin{eqnarray*}
	\theta_* =	\frac{-1 + \sqrt{1+ 4 \norm{c_t}_2}}{2} \qquad \text{for } r = 1 . 
\end{eqnarray*}
and obtain the following method
\begin{eqnarray*}
	c_t &=& x_t - \stepsize_{t} \nabla f(x_t, \xi_t) + \norm{x_t}_2 x_t , \\
	x_{t+1} &=& \frac{2 c_t}{1 + \sqrt{1 + 4 \norm{c_t}_2}} .
\end{eqnarray*}

For $r = 2$, an explicit form solution can be written using Cardano's formula. We use Python \texttt{Sympy} library for symbolic calculation to solve for $\theta$ in this case.

More generally, \eqref{eq:derivation_of_SMDr} implies for any $r > 0$, we have 
\begin{eqnarray*}
c_t  &= &(1 + \norm{x}_2^r) \, x_t - \stepsize_{t} \nabla f(x_t, \xi_t ) , \\
x _{t+1} &=&\frac{c_t}{1 + \theta_*^r} ,
\end{eqnarray*}
where $\theta_*$ is the solution to \eqref{eq:univ_equation}, which can be solved using a bisection method up to the machine accuracy. 

\begin{corollary}\label{cor:SMD1}
	Let $F(\cdot): \R^d \rightarrow \R$ be twice differentiable and satisfy \eqref{eq:L_Lr_condition}. Let Assumption~\ref{ass:BV} hold with $\norm{\cdot}_2$. Suppose the sequence $\cb{\eta_t}_{t\geq0}$ be non-increasing with $\eta_0 \leq 1/(2\ell)$, and $\bar{x}_T \in \cX$ be randomly chosen from the iterates $x_0, \ldots, x_{T-1}$ with probabilities $p_t = \stepsize_t / \sum_{t=0}^{T-1}\stepsize_t$. Then for \eqref{eq:SMDr_line1}, \eqref{eq:SMDr_line2}, we have 
	\begin{equation*}\label{eq:SMDr1}
		\Exp{ \sqnorm{\nabla F(\bar x_T)}_2 } \leq \frac{6 (F(x_0) - F^*) + 6 \ell \sigma^2  \sum_{t=0}^{T-1} \eta_t^2 }{ \sum_{t=0}^{T-1} \eta_t } , 
	\end{equation*}
where $F^* := \min_{y\in \R^d} F(y)$.
\end{corollary}  

\textbf{Additional experimental details.} We use \texttt{Fashion-MNIST} dataset \citep{xiao2017fashion} for training with images of dimensions $d_f = 28 \times 28 = 784$.  The encoding dimension is fixed to $d_e = 64$. The dataset is of size $50000$ images. In all experiments, we use the mini-batch of size $100$. We initialize the parameters $W$ of the model with a normal distribution with mean $1$ and the standard deviation $0.01$. 
 
  \begin{remark}
	A momentum variant of the scheme \eqref{eq:SMDr_line1}, \eqref{eq:SMDr_line2} was recently explored in \citep{ding2023_NC_MSBPG} with promising empirical results on image classification and language modeling tasks. We hope that our simpler variant without momentum can be also helpful in these tasks. 
\end{remark}

\textbf{Additional discussion about $(L_0, L_1)$-smoothness.}
Recently, some works, e.g.,  \citep{zhang2019gradient_clipping_L0L1,faw2023_adagrad_L0L1,hubler2023parameter}, consider adaptive gradient methods such as gradient clipping, AdaGrad-Norm and gradient normalization under $(L_0, L_1)$-smoothness, i.e., $F(\cdot)$ is twice differentiable and for some $L_0, L_1 \geq 0$ satisfies $\norm{\nabla F(x)}_{\text{op}} \leq L_0 + L_1 \norm{\nabla F(x)}_2$ for all $x\in \R^d$. The authors in \citep{zhang2019gradient_clipping_L0L1,faw2023_adagrad_L0L1} justify the theoretical benefits of the popular adaptive schemes by the fact that, unlike \algname{SGD}, they provably work under this weaker $(L_0, L_1)$-smoothness. Moreover, \citet{zhang2019gradient_clipping_L0L1} empirically verify that $(L_0, L_1)$-smoothness condition holds on the optimization trajectory when training modern language and image classification models. Our polynomial grow condition is weaker than $(L_0, L_1)$-smoothness as long as the gradient norm grows at most as a polynomial in $\norm{x}_2$. Unlike the approach taken in the above mentioned works, the convergence of our algorithm with the choice of DGF as in Proposition~\ref{prop:polygrow_hess} follows directly from Theorem~\ref{thm:SMD} and does not require a separate analysis.

\newpage
\section{Useful Lemma}

The following lemma is standard \citep{lu2018relatively} and the proof can be found, e.g., in \citep{chen1993convergence}.
\begin{lemma}\label{le:bregman_properties}
 \begin{enumerate}
     \item
    The Bregman divergence satisfies the three-point identity:
    $$
    D_{\omega}(x,y) + D_{\omega}(y,z) = D_{\omega}(x,z) + \langle \nabla \omega(z) - \nabla \omega(y), x - y \rangle  \quad \text{for all } \, y, z \in \cS \text{ and } x \in \text{cl}(\cS)  .
    $$ \label{le:bregman_properties_TPI}
    \item Let $\phi(\cdot)$ be a closed proper convex function on $\R^d$, $z \in \cS$ and $z^{+} := \argmin_{x\in \cX} \left\{ \phi(x) + \rho D_{\omega} (x, z) \right\} $ for $\rho>0$, then
    $$
    \phi(x) + \rho D_{\omega}(x, z) \geq \phi(z^{+}) + \rho D_{\omega}(z^{+}, z) + \rho D_{\omega}(x, z^{+} )  \quad \text{for all } x \in \text{cl}(\cS) . 
    $$\label{le:bregman_properties_optimality}
\end{enumerate}
 \end{lemma}
To establish high probability convergence, we use the technical lemma by \citet{liu2023high}. 
\begin{lemma}[Lemma 2.2. in \citep{liu2023high}]\label{le:lemma22_liu}
	Suppose $X \in \mathbb{R}^d$ such that $\mathbb{E}[X]=0$ and $\|X\|_*$ is a $\sigma$-sub-Gaussian random variable, then for any $a \in \mathbb{R}^d, 0 \leq b \leq \frac{1}{2 \sigma}$,
	$$
	\mathbb{E}\left[\exp \left(\langle a, X\rangle+b^2\|X\|_*^2\right)\right] \leq \exp \left(3\left(\|a\|^2+b^2\right) \sigma^2\right) .
	$$
\end{lemma}

 The following lemma shows the connection between $\Phi_{1/\rho}$ and $\Phi$. Similar result in the Euclidean setting has previously appeared, e.g., in \citep{stella2017forward}.
 \begin{lemma}\label{le:MEnv_Phi_connection}
 	Let $F(\cdot)$ be $(\ell , \omega)$-smooth. Then for any $\rho \geq 2 \ell$ and $x\in \cX \cap \cS$ we have $\Phi_{1/\rho}(x) \geq \Phi(x^+)$, where $ x^+ := \argmin_{y \in \cX} \, \langle \nabla F(x), y \rangle + r(y) + (\rho - \ell ) D_{\omega}(y, x) . $
 \end{lemma}
\begin{proof}
	By Assumption~\ref{ass:rel_smooth} (lower bound), we have for any $x,y \in \cX \cap \cS$
	$$
	\Phi(y) + \rho D_{\omega}(y, x) \geq  F(x) + \langle \nabla F(x), y - x\rangle + r(y) + (\rho - \ell ) D_{\omega}(y , x ) .
	$$
	Minimizing both sides over $y \in \cX \cap \cS$, we have 
	\begin{eqnarray*}
	\Phi_{1/\rho}(x) &\geq& F(x) + \langle \nabla F(x), x^+ - x\rangle + r(x^+) + (\rho - \ell ) D_{\omega}(x^+, x ) \\
	&\geq& F(x^+) + r(x^+) + (\rho - 2 \ell ) D_{\omega}(x^+, x )  \geq \Phi(x^+)  ,
		\end{eqnarray*}
		where the first equality holds by the definitions of $\Phi_{1/\rho}$ and $x^+$. The second inequality uses Assumption~\ref{ass:rel_smooth} (upper bound). 
\end{proof}
The following lemma shows that our Assumption~\ref{ass:KL} is more general than relative strong convexity \citep{lu2018relatively}. In the Euclidean case, the same result was derived by \citet{Karimi_PL}.
\begin{lemma}[Relative strong convexity implies $2$-Bregman Prox-P{\L}]\label{le:rel_SC_Prox_PL}
	Let $F(\cdot)$ be $\mu$-relatively strongly convex w.r.t. $\omega(\cdot)$, i.e., for all $x, y \in \cX \cap \cS$
	\begin{eqnarray}\label{eq:rel_SC}
	F(y) \geq F(x) + \langle \nabla F(x), y - x \rangle +  \mu D_{\omega}(y,x) .
	\end{eqnarray}
	Then Assumption~\ref{ass:KL} holds wth $\alpha = 2$ and any $\rho \geq \mu$, i.e.,  $ \cD_{\rho}(x)  \geq  2 \mu \rb{ \Phi(x) - \Phi^* }.$
\end{lemma}
\begin{proof}
	Adding $r(y) $ to both sides of \eqref{eq:rel_SC}, we have 
	$$
	\Phi(y) \geq \Phi(x) + \langle \nabla F(x), y - x \rangle +  \mu D_{\omega}(y,x)  + r(y) - r(x)  = \Phi(x) + Q_{\mu}(x, y) .
	$$
	Minimizing both sides over $y\in \cX \cap \cS$, we get
	$$
	\Phi^* \geq \Phi(x) + \min_{y\in \cX} Q_{\mu} (x, y) = \Phi(x) - \frac{1}{2 \mu}  \cD_{\mu}(x) . 
	$$
	Rearranging and noticing that $\cD_{\mu}(x) \leq \cD_{\rho}(x)$ for any $x\in \cX$ and $\rho \geq \mu$, we obtain the result. 
\end{proof}

The following lemma connects the Frank-Wolfe gap with the norm of the gradient mapping in the Euclidean case. 

\begin{lemma}[Lemma 2.2 in \citep{balasubramanian2022zeroth}]\label{le:FWGap_BGM}
	Let $\omega(x) := \frac{1}{2}\sqnorm{x}_2$, $\cX$ be a compact set with diameter $D_{\cX, \norm{\cdot}_2 } := \max_{x, y\in \cX} \norm{x - y}_2$ and $r(\cdot) = 0$. Then for any $\rho > 0 $ 
	$$
	\max_{y \in \cX} \langle \nabla F(x) , x - y \rangle \leq \rb{ D_{\cX, \norm{\cdot}_2} + \rho^{-1} G_{F, \norm{\cdot}_{2}} } \sqrt{\Delta_{\rho}^+(x)} ,
	$$
	where $G_{F, \norm{\cdot}_{2}}  := \max_{x\in \cX} \norm{\nabla F(x)}_2$.
\end{lemma}

We report the special case of Lemma 3 by \citet{stich2019unified}. 
\begin{lemma}[Lemma 3 in \citep{stich2019unified}]\label{le:L3_noise_adaptive_rate}
Let $\cb{r_t}_{t\geq0}$ and $\cb{\stepsize_{t}}_{t\geq0}$ be two non-negative sequences with $\stepsize_{t} \leq \frac{1}{d}$ that satisfy the relation 	
$$
r_{t+1} \leq (1 - a \stepsize_t) \, r_t + c \, \stepsize_{t}^2 ,
$$
where $a > 0$, $c \geq 0$. For any $T\geq0$, set
$$
\stepsize_{t} = \begin{cases}
	\frac{1}{d}  & \text{if } t < \lceil T/2 \rceil \text{ and } T \leq \frac{2 d }{a } , \\
	\frac{1}{a \rb{ \frac{2 d}{a} + t - \lceil T/2 \rceil  }}   & \text{otherwise.} 
\end{cases}
$$
Then we have 
$$
 r_{t+1} \leq \frac{32 \,  d \, r_0 }{a} \exp\rb{- \frac{a T}{2 d}} + \frac{36 \, c}{a^2 T} . 
$$
\end{lemma}

The next lemma is standard and the proof can be found, e.g., in \citep{van2014probability}. 
\begin{lemma}[Maximal tail inequality, Lemma 5.1 and 5.2 in \citep{van2014probability}]\label{le:max_tail}
	Let $\xi_i$ be a $\sigma$-sub-Gaussian random variable for every $i = 1, \ldots, n$. Then 
	$$
	\rb{\Exp{\max_{1\leq i \leq n} \xi_i} }^2 \leq \Exp{\max_{1\leq i \leq n} \xi_i^2} \leq 2 \sigma^2 \log(n) , 
	$$
	$$
	\opn{Pr}\rb{\max_{1\leq i \leq n} \xi_i \geq \sqrt{2 \sigma^2 \log(n)}  + \lambda } \leq e^{-\frac{\lambda^2}{2 \sigma^2} } \qquad \text{for all } \lambda \geq 0 . 
	$$
\end{lemma}

\end{document}